%% file: jomp_final.tex
\documentclass[smallextended]{svjour3}     

\usepackage[colorlinks=true, linkcolor=black, citecolor=black, urlcolor=black, pdfborder={0 0 0}, backref]{hyperref}
\usepackage[letterpaper,top=2in, bottom=1.5in, left=1in, right=1in]{geometry}
\usepackage{tabu}
\usepackage{makecell}
\usepackage{multirow}

\input{macros} 
\usepackage{tikz} 
\newtheorem{coro}{Corollary}
\newtheorem{prop}{Proposition}

\newcommand{\norm}[1]{{\| #1 \|}} 

\newcommand{\Lag}{{\mathcal L}}
\title{Global Convergence of ADMM in Nonconvex Nonsmooth Optimization \thanks{The work of W. Yin was supported in part by NSF grants DMS-1720237 and ECCS-1462397, and ONR Grant N00014171216. The work of J. Zeng was supported in part by the NSFC grants (61603162, 11501440, 61772246, 61603163) and the doctoral start-up foundation of Jiangxi Normal Univerity.
}}

\author{Yu Wang \and
        Wotao Yin \and
        Jinshan Zeng$^\dag$}

\institute{
Y. Wang\at
              Department of Statistics, University of California, Berkeley (UCB),
              Berkeley, CA 94704, USA\\
              \email{wang.yu@berkeley.edu}\\[3pt]
W. Yin\at
              Department of Mathematics, University of California, Los Angeles (UCLA),
              Los Angeles, CA 90025, USA\\
              \email{wotaoyin@ucla.edu}\\[3pt]
Corresponding author: J. Zeng\at
             School of Computer and Information Engineering, Jiangxi Normal University, Nanchang,  Jiangxi 330022, China\\
              \email{jsh.zeng@gmail.com}}

\date{\today}
\journalname{Report} 
\usepackage{amsbsy}
\usepackage[normalem]{ulem}

\begin{document}

\maketitle
\abstract{
In this paper, we analyze the convergence of the alternating direction method of multipliers (ADMM) for minimizing a nonconvex and possibly nonsmooth objective function, $\phi(x_0,\ldots,x_p,y)$, subject to coupled linear equality constraints. Our ADMM updates each of the primal variables $x_0,\ldots,x_p,y$, followed by updating the dual variable. We separate the variable $y$ from $x_i$'s as it has a special role in our analysis.

The developed convergence guarantee covers a variety of nonconvex functions such as piecewise linear functions,  $\ell_q$ quasi-norm, Schatten-$q$ quasi-norm ($0<q<1$), minimax concave penalty (MCP), and smoothly clipped absolute deviation (SCAD) penalty. It also allows nonconvex constraints such as compact manifolds (e.g., spherical, Stiefel, and Grassman manifolds) and linear complementarity constraints. Also, the $x_0$-block can be almost any lower semi-continuous function.

By applying our analysis, we show, for the first time, that several ADMM algorithms applied to solve nonconvex models in statistical learning, optimization on manifold, and matrix decomposition are guaranteed to converge.

Our results provide sufficient conditions for ADMM to converge on (convex or nonconvex) monotropic programs with three or more blocks, as they are special cases of our model.

ADMM has been regarded as a variant to the augmented Lagrangian method (ALM). We present a simple example to illustrate how ADMM converges but ALM diverges with bounded penalty parameter $\beta$. Indicated by this example and other analysis in this paper, ADMM might be a better choice than ALM for some nonconvex \emph{nonsmooth} problems, because ADMM is not only easier to implement, it is also more likely to converge for the concerned scenarios.
}
\section{Introduction}
In this paper, we consider the (possibly nonconvex and nonsmooth) optimization problem:

\begin{align}\label{Eq:originalForm0}
  \Min_{x_0,x_1,\ldots,x_p,y}\quad & \hspace{12pt}\phi(x_0,x_1,\ldots,x_p,y)\\\nonumber
  \St\quad & A_0x_0 + A_1x_1 +\cdots+A_px_p + B y = b,
\end{align}
where $\phi:\RR^{n_0}\times\cdots\times\RR^{n_p}\times \RR^q\to \RR\cup\{\infty\}$ is a continuous function, $x_i\in\RR^{n_i}$ are variables with their coefficient matrices $A_i\in\RR^{m\times n_i}$,  $i=0,\ldots,p$, and $y\in\RR^q$ is the last variable with its coefficient matrix $B\in\RR^{m\times q}$. The model remains as general without $y$ and $By$; but we keep $y$ and $B$ to simplify the notation.

We set $b=0$ throughout the paper to simplify our analysis. All of our results still hold if $b\not=0$ is in the image of the matrix $B$, i.e., $b\in \mathrm{Im}(B)$.

Besides the linear constraints in \eqref{Eq:originalForm0}, any constraint on each variable $x_0,x_1,\ldots,x_p$ and $y$ can be treated as an indicator function and included in the objective function $\phi$. Therefore, we do not include constraints like: $x_0\in \mathcal{X}_0, x_1\in \mathcal{X}_1,\ldots, x_p\in \mathcal{X}_p, y\in \mathcal{Y}.$


In spite of the success of ADMM on convex problems, the behavior of ADMM on nonconvex problems has been largely a mystery, especially when there are also nonsmooth functions and nonconvex sets in the problems. 
ADMM  generally fails on nonconvexity problems, but it has found to not only work in some applications but often exhibit great performance! Indeed, 
successful examples include: matrix completion and separation \cite{Yang2015,xu11,shen11,Sun2014}, asset allocation \cite{wen13}, 
tensor factorization \cite{liavas2014},
phase retrieval \cite{Wen2012}, compressive sensing \cite{Chartrand2013}, optimal power flow \cite{You2014}, direction fields correction \cite{Lai2014}, noisy color image restoration \cite{Lai2014}, image registration \cite{Bouaziz2013}, network inference \cite{Miksik2014}, and global conformal mapping \cite{Lai2014}.
In these applications, the objective function can be nonconvex, nonsmooth, or both. Examples include the piecewise linear function, the $\ell_q$ quasi-norm for $q\in(0,1)$, the Schatten-$q~(0<q<1)$ \cite{wiki:xxx} quasi-norm $f(X) = \sum_{i} \sigma_i(X)^q$ (where $\sigma_i(X)$ denotes the $i$th largest singular value of $X$), and the indicator function $\iota_{\cB}$, where $\cB$ is a nonconvex set.

The success of these applications can be intriguing, since these applications are far beyond the scope of the theoretical conditions that ADMM is proved to converge.
In fact, even the three-block ADMM can diverge on a simple convex problem \cite{Chen2014}. Nonetheless, we still find that it  works well in practice. This has motivated us 
to explore in the paper and respond to this question: when will the ADMM type algorithms converge if the objective function includes nonconvex nonsmooth functions?

We present our Algorithm \ref{Alg:BCD}, where $\Lag_\beta$ denotes the augmented Lagrangian (\ref{Eq:lagrangian}), and show that it converges for a large class of problems.
For simplicity, Algorithm \ref{Alg:BCD}  uses the standard ADMM subproblems, which minimize the augmented Lagrangian $\Lag_\beta$ with all but one variable fixed.
It is possible to extend them to inexact, linearized, and/or prox-gradient subproblems as long as a few key principles (cf. \S 3.1) are preserved.


In this paper, under some assumptions on the objective and matrices, Algorithm \ref{Alg:BCD} is proved to converge. 
Algorithm \ref{Alg:BCD}  is a generalization to the coordinate descent method. 
By setting $A_0,A_1,\ldots,A_p,B$ to 0, Algorithm \ref{Alg:BCD}  reduces to the \emph{cyclic} coordinate descent method. 

\subsection{Proposed algorithm}
Our variable is $\vx:=[x_0;\ldots;x_p]\in\RR^n$ where $n = \sum_{i=0}^p n_i$. Let $\vA:=[A_0~\cdots~A_p]\in\RR^{m\times n}$ and $\vA\vx:=\sum_{i=0}^pA_ix_i\in\RR^{m}$. To present our algorithm, we define the augmented Lagrangian:
\begin{align}\label{Eq:lagrangian}
\Lag_\beta(\vx,y,w) & := \phi(\vx,y) + \dotp{w,\vA\vx+By}+ \frac{\beta}{2}\|\vA\vx+By \|^2.
\end{align}
The proposed Algorithm \ref{Alg:BCD} extends the standard ADMM to  multiple variable blocks. It also extends the \emph{coordinate descent} algorithms dealing with linear constraints. We let $x_{<i} := [x_0;\ldots;x_{i-1}]\in \RR^{n_0+n_1+\cdots+n_{i-1}}$ and $x_{>i} := [x_{i+1};\ldots;x_{p}]\in \RR^{n_{i+1}+\cdots+n_{p}}$ (clearly, $x_{<0}$ and $x_{>p}$ are null variables, which may be used for notational ease). Subvectors $x_{\leq i} := [x_{<i}, x_i]$ and $x_{\geq i}$ are defined similarly. The convergence of Algorithm \ref{Alg:BCD} will be given in Theorems \ref{thm:main} and \ref{thm:main_theorem_convex}.

\begin{algorithm}[t]
{\small
\begin{algorithmic}\caption{Nonconvex ADMM for \eqref{Eq:originalForm0}}\label{Alg:BCD}
\STATE {\bf Initialize} $x_1^0,\ldots,x_p^0,y^0,w^0$ 
\smallskip
\WHILE{stopping criteria not satisfied}
\smallskip
\FOR{$i=0,\ldots,p$} 
\STATE $x_i^{k+1} \gets \argmin_{x_i} \Lag_\beta(x^{k+1}_{<i},x_i,x_{>i}^k,y^k,w^k)$;
\ENDFOR
\smallskip
\STATE $y^{k+1} \gets \argmin_{y} \Lag_\beta(\vx^{k+1},y,w^k)$;
\smallskip
\STATE $w^{k+1} \gets w^k + \beta\left(\vA\vx^{k+1}+By^{k+1}\right)$;
\smallskip
\STATE $k\gets k+1$;
\ENDWHILE
\STATE \textbf{return} $x_1^k,\ldots,x_p^k$ and $y^k$.
\end{algorithmic}}
\end{algorithm}

\subsection{Relation to the augmented Lagrangian method (ALM)}
ALM is a widely-used method for solving constrained optimization models \cite{Hestenes1969,Powell1967}. It applies broadly to nonconvex nonsmooth problems.
ADMM is an approximation to ALM by 
sequentially updating each of the primal variables.

ALM  generally uses a sequence of  penalty parameters $\{\beta^k\}$, which is nondecreasing and possibly unbounded. When $\beta^k$ becomes large, the ALM subproblem becomes ill-conditioned. Therefore,  using bounded  $\beta^k$ is practically desirable (see \cite[Theorem 5.3]{Conn1991}, \cite[Proposition 2.4]{Bertsekas2014}, or \cite[Chapter 7]{Birgin2014}).
For general nonconvex and nonsmooth problems, it is well known that
$\beta^k$, $k\in \mathbb{N}$ is bounded is not enough for the convergence of ALM \cite[Section 2.1]{Bertsekas2014}.  Proposition \ref{prop:convergence} below introduces a  simple example on which ALM diverges with any bounded $\beta^k$. 
It is surprising, however, that ADMM converges in finite steps for any fixed $\beta>1$ on this example.
\begin{proposition}\label{prop:example}
Consider the problem
\begin{align}\label{Eq:specificEg}
\Min_{x,y\in\RR} & \quad x^2 - y^2\\
\St & \quad x = y,~ x\in[-1,1].\nonumber
\end{align}
It holds that
\begin{enumerate}
\cut{\item the Lagrangian of \eqref{Eq:specificEg} is unbounded; the augmented Lagrangian $\Lag_\beta(x,y,w)$ is bounded;}
\item If $\{\beta^k|k\in\mathbb{N}\}$ is bounded, ALM generates a divergent sequence;
\item for any fixed $\beta>4$, ADMM generates a convergent and finite sequence to a solution.
\end{enumerate}
\end{proposition}
The proof is straightforward and included in the Appendix\cut{ for completeness}.  ALM diverges because $\Lag_\beta(x,y,w)$ does not have a saddle point, and there is a non-zero duality gap.\cut{, which means the duality gap of the model is not zero. The Gauss-Seidel update in ADMM makes it much more stable in this example.} ADMM, however, is unaffected. As the proof shows, the ADMM sequence satisfies  $2y^k=-w^k$, $\forall k$. By substituting $w\equiv-2y$ into $\Lag_\beta(x,y,w)$,\cut{ from the augmented Lagrangian, then} we get a convex function in $(x,y)$! Indeed,
\begin{align*}
\rho(x,y):=\Lag_\beta(x,y,w)\big|_{w=-2y} & = (x^2 - y^2)+\iota_{[-1,1]}(x)-2y(x-y)+\frac{\beta}{2}\big|x-y\big|^2= \frac{\beta+2}{2}|x-y|^2+\iota_{[-1,1]}(x),
\end{align*}
where $\iota_{S}$ is the indicator function of set $S$ (that is, $\iota_S(x)=0$ if $x\in S$; otherwise, equals infinity).
It turns out that ADMM solves \eqref{Eq:specificEg} by performing the following  coordinate descent iteration to $\rho(x,y)$:
$$
\begin{cases}
x^{k+1} = \argmin_x \rho(x,y^k), \\
y^{k+1} = y^k - \frac{\beta}{\beta^2 - 4} \frac{\mathrm{d}}{\mathrm{d}y}\rho(x^{k+1},y^k).
\end{cases}
$$
Our analysis for the general case will show that  the primal variable $y$ somehow  ``controls" 
the dual variable $w$ and reduces ADMM to an iteration that is similar to coordinate descent.

\subsection{Related literature}\label{sec:literature}

The original ADMM was proposed in \cite{Glowinski1975On,gabay1976}. For convex problems, its convergence was established firstly in \cite{Glowinski1984Numerical} and its convergence rates given in \cite{He2012On,Davis2014Convergencea,Davis2014Convergencec} in different settings.
When the objective function is nonconvex, the recent results \cite{xu11,jiang13,Magnusson2014} directly make assumptions  on the iterates $(\vx^k,y^k,w^k)$.
Hong et al.  \cite{hong2014a} deals with the nonconvex separable objective functions for some specific $A_i$,
which forms the sharing and consensus problem.
Li and Pong \cite{Li2014} studied the convergence of ADMM for some special nonconvex models, where one of the matrices $A$ and $B$ is an identity matrix. Wang et al. \cite{Wang2015,Wang2014} studied the convergence of the nonconvex Bregman ADMM algorithm, which includes ADMM as a special case. We review their results and compare to ours in \S \ref{sec:discussion} below.

\subsection{Contribution and novelty}

The main contribution of this paper is the establishment of the global convergence of Algorithm \ref{Alg:BCD} under certain assumptions given in Theorems \ref{thm:main} and \ref{thm:main_theorem_convex} below. 
The assumptions apply to largely many nonconvex and nonsmooth objective functions.
The developed theoretical results can be extended to the case where subproblems are solved inexactly with summable errors.
We also allow the primal block variables $x_1,\ldots,x_p$ to be updated in an arbitrary order as long as $x_0$ is updated first and $y$ is updated last (just before the $w$-update).
The novelty of this paper can be summarized as follows:
\begin{enumerate}
\item[(1)] {\bf Weaker assumptions.} Compared to the related works \cite{xu11,jiang13,Magnusson2014,hong2014a,Li2014,Wang2015,Wang2014}, the convergence conditions in this paper are  weaker, extending the ADMM theory to significantly more nonconvex functions and nonconvex sets.  See Table 1. In addition, we allow the primal variables $x_1,\ldots, x_p$ to be updated in an arbitrary order at each iteration\footnote{This is  the best that one hope (except for very specific problems) since \cite[Section 1]{yy2016} shows a convex 2-block problem, which ADMM fails to converge. }, which is new in the ADMM literature. We  show that most of our assumptions are necessary by providing counter examples. We also give the first example that causes ADMM to converge but ALM to diverge.

\item[(2)] {\bf New examples.} By applying our main theorems, we prove convergence for the nonconvex ADMM applied to the following problems which could not be recovered from previous convergence theory:
\begin{itemize}
  \item statistical regression based on nonconvex regularizer such as minimax concave penalty(MCP), smoothly clipped absolute deviation (SCAD), and  $\ell_q$ quasi-norm;
  \item minimizing smooth functions subject to norm or Stiefel/Grassmannian manifold constraints;
  \item matrix decomposition using nonconvex Schatten-$q$ regularizer;
  \item smooth minimization subject to complementarity constraints.
\end{itemize}
\item[(3)] {\bf Novel techniques.}
We  improve upon the existing analysis techniques and introduce new ones.

    \begin{enumerate}
    \item[(a)] \textit{An induction technique for nonconvex, nonsmooth case.}
    The analysis uses the augmented Lagrangian as the Lyapunov function: Algorithm \ref{Alg:BCD} produces a sequence of points whose augmented Lagrangian function values are decreasing and lower bounded. This technique appeared first in \cite{hong2014a} and also in \cite{Li2014,Wang2015}. However, it has trouble handling nonsmooth functions. An induction technique is introduced to overcome this difficulty and extend the current framework to nonconvex, nonsmooth, multi-block cases. The technique is used in the proof of Lemma \ref{lm:p2}.

    \item[(b)] \textit{Restricted prox-regularity.} Most of the convergence analysis of nonconvex optimization either assumes or proves the sufficient descent and bounded subgradient properties (c.f., \cite{attouch2011,hong2014a}). This property is easily obtainable if the objective is smooth. However, some nonconvex and nonsmooth objectives (e.g. nonconvex $\ell_q$ quasi-norm) violate these properties. We overcome this challenge with the introduced \emph{restricted prox-regularity property} (Definition \ref{def:quadcvx}). If the objective satisfies such a property, we prove that the sequence enjoy sufficient descent and bounded subgradients after a finite number of iterations.

    \item[(c)] \textit{More general linear mappings}. Most nonconvex ADMM  analysis is applied to the primal variables $\vx$ and $y$ directly. 
This requires the matrices $A_0,A_1,\ldots,A_p,B$ to either identity or have full column/row rank. In this paper, we introduce techniques to work with possibly rank-deficient $A_0,A_1,\ldots,A_p,B$ 
(see, for example, Lemma \ref{lm:yw}). This allows us to ensure convergence of ADMM on some important applications in signal processing and statistical learning (see \S \ref{sec:application}).
    \end{enumerate}
  In addition, we use several other techniques that are tailored to relax our convergence assumptions as much as possible.
\end{enumerate}

\subsection{Notation and organization}
We denote $\mathbb{R}$ as the real number set, $\mathbb{R}\cup \{+\infty\}$ as the extended real number set, $\mathbb{R}_+$ as the positive real number set, and $\mathbb{N}$ as the natural number set. Given a matrix $X$, $\mathrm{Im}(X)$ denotes its image, $\sigma_i(X)$ denotes its $i$th largest singular value. $\|\cdot\|$ represents the Euclidean norm for a vector or the Frobenius norm for a matrix. $\mathrm{dom}(f)$ denotes the domain of a function $f$. For any two square matrices $A$ and $B$ with the same size, $A \succeq B$ means that $A-B$ is positively semi-definite.

The remainder of this paper is organized as follows. Section \ref{sec:result} presents the main convergence analysis. Section \ref{sec:proof} gives the detailed proofs. Section \ref{sec:discussion} discusses the tightness of the assumptions, the primal variable update order, and inexact minimization issues. Section \ref{sec:application} applies the developed theorems in some typical applications and obtains novel convergence results. Finally, Section \ref{sec:conclusion} concludes this paper.
\begin{table}[ht!]
\tabulinesep=1.5mm
\centering
\begin{center}
\caption{Conditions for ADMM convergence (note: $f_0,f_1,\ldots,f_p$ is not required to exist)}
\end{center}

\vspace{1mm}
\label{tab:main_results}
\begin{tabu}{l|cc|ll}\hline
                  & \multicolumn{2}{c|}{Scenario 1}                                            & \multicolumn{2}{c}{Scenario 2}                                               \\ \hline
                 model & \multicolumn{2}{l|}{$\begin{array}{cc}
  \Min\limits_{\vx=(x_0,\ldots,x_p),y}~&\phi(\vx, y) := g(\vx) + \sum_{i=0}^p f_i(x_i) + h(y)\\
  \St & \vA\vx + B y = 0
\end{array}$}          & \multicolumn{2}{l}{$\begin{array}{cl}
  \Min\limits_{\vx,y}\quad & \phi(\vx,y)\\
  \St & \vA\vx + B y = 0
\end{array}$}               \\ \hline
$\phi$            & \multicolumn{4}{c}{coercive over the feasible set $\{(\vx,y):\vA\vx + B y = 0\}$; see assumption \ref{A_coercive}}
    \\ \hline
$g$, $h$          & \multicolumn{2}{c|}{Lipschitz differentiable}                              &                                                                       \multicolumn{2}{l}{\multirow{4}{*}{$\phi$ Lipschitz differentiable}}      \\ \cline{1-3}
             & \multicolumn{1}{c|}{Scenario 1a}                 & Scenario 1b   &     \multicolumn{2}{l}{}                                     \\ \cline{1-3}
$f_0$             & \multicolumn{1}{c|}{lower semi-continuous}                 & $\partial f$  bounded in any bounded set   &     \multicolumn{2}{l}{}                                     \\ \cline{1-3}
$f_1,\ldots, f_p$ & \multicolumn{1}{c|}{restricted prox-regular}   & piecewise linear                              &                   \multicolumn{2}{l}{}                                                         \\  \hline
$\vA, B$            & \multicolumn{4}{c}{$\image{\vA} \subseteq \image{B}$}                                                                                                                \\ \hline
                  & \multicolumn{4}{c}{
solution to each ADMM sub-problem is Lipschitz w.r.t. input (A3)
} \\ \hline
\end{tabu}
\end{table}

\section{Main results}\label{sec:result}
\subsection{Definitions}
In our definitions, $\partial f$ denotes the  set of general subgradients of $f$ in \cut{Definition \ref{def:general_sub_gradient} (c.f. }\cite[Definition 8.3]{VariationalAnalysis-RockWets}.
We call a function \emph{Lipschitz differentiable} if it is differentiable and its gradient is Lipschitz continuous. The functions given in the next two definitions are permitted in our model.


\begin{definition}[Piecewise linear function]
\label{def:piecewise_linear}
A function $f:\RR^n\to\RR$ is  \emph{piecewise linear} if there exist polyhedra $U_1,\ldots,U_K\subset\RR^n$, vectors~$a_1,\ldots,a_K\in\RR^n$, and points~$b_1,\ldots,b_K\in\RR$ such that $\bigcup_{i=1}^K \overline{U_i} = \RR^n$, $U_i\bigcap U_j = \emptyset~ (\forall~i\neq j)$, and
$f(x) = a_i^Tx + b_i~ \text{when } x\in U_i$, $i=1,\ldots,K$.
\end{definition}

\begin{definition}[Restricted prox-regularity]\label{def:quadcvx}
For a lower semi-continuous function $f$, let  $M\in\RR_+$, $f:\RR^N\to\RR\cup\{\infty\}$, and define the exclusion set
$$S_M:=\{x\in\dom(f): \|d\|>M~\mbox{for all}~d\in\partial f(x)\}.$$
$f$ is called \emph{restricted prox-regular} if, for any $M > 0$ and bounded set $T\subseteq \mathrm{dom}f$, there exists $\gamma>0$ such that
\begin{align}\label{quadcvx}
f(y)+\frac{\gamma}{2}\|x-y\|^2\ge f(x)+\dotp{d,y-x},\quad \forall~x\in T\setminus S_M,~y\in T ,~d\in\partial f(x),~\|d\|\le M.
\end{align}
(If $ T\setminus S_M$ is empty, \eqref{quadcvx} is satisfied.)
\end{definition}
Definition \ref{def:quadcvx} is related to, but weaker than, the concepts \emph{prox-regularity} \cite{Poliquin1996}, \emph{hypomonotonicity} \cite[Example 12.28]{VariationalAnalysis-RockWets} and \emph{semi-convexity} \cite{Mifflin1977,Ivanov1997,Krystof2015,Mollenhoff2015}, all of which impose global conditions. Definition \ref{def:quadcvx} only requires \eqref{quadcvx} to hold over a subset. As shown in Proposition \ref{Prop:restricted_prox}, while prox-regular functions include any convex functions and any $C^1$ functions with Lipschitz continuous gradients, restricted prox-regular functions further include a set of non-smooth non-convex functions such as $\ell_q$ quasi-norms ($0<q<1$), Schatten-$q$ quasi-norms ($0<q<1$), and indicator functions of compact smooth manifolds.
\begin{prop}\label{Prop:restricted_prox}\textbf{Examples of (restricted) prox-regular functions} The following functions are prox-regular functions:
\begin{itemize}
\item[(1)] convex functions, including indicator functions of convex sets,
\item[(2)] $C^1$ smooth functions with $L$-Lipschitz continuous gradient.
\end{itemize}
The following functions are restricted prox-regular functions:
\begin{itemize}
\item[(1)] $\ell_q(x) := \|x\|_q^q$ function for $q\in (0, 1)$;
\item[(2)] Schatten-q quasi-norm:
$\norm{A}_{q} = \sum_{i=1}^n \sigma_i^{q}$,
where $q\in(0,1)$ and $\sigma_i$ is the $i$th largest singular value of $A$;
\item[(3)] Indicator functions $\iota_S$ of a compact $C^2$ manifold, such as the unit sphere in a finite Euclidean space.
\end{itemize}
\end{prop}

Definition \ref{def:quadcvx} introduces functions that do not satisfy \eqref{quadcvx} globally \emph{only because} they are asymptotically ``steep" in the exclusion set  $S_M$.
Such functions include $|x|^q$ ($0<q<1$), for which $S_M$ has the form $(-\epsilon_M,0)\cup(0,\epsilon_M)$; the Schatten-$q$ quasi-norm ($0<q<1$), for which $S_M = \{X: \exists i, \ \sigma_i(X) <\epsilon_M\}$ as well as $\log(x)$, for which $S_M=(0,\epsilon_M)$, where $\epsilon_M$ is a constant depending on $M$. We only  need \eqref{quadcvx} 
because the iterates $x_i^k$ of Algorithm \ref{Alg:BCD}, for all large $k$, never enter the exclusion set $S_M$. 
\subsection{Main theorems}
To ensure the boundedness of the sequence $(\vx^k,y^k,w^k)$, we only need the coercivity of the objective function within the feasible set.
\begin{enumerate}[label={A\arabic*}]
\item\label{A_coercive} \textbf{(coercivity)} Define the feasible set $\cF:=\{(\vx,y)\in\RR^{n+q}:\vA\vx+By=0\}$. The objective function $\phi(\vx,y)$ is coercive over this set, that is,  $\phi(\vx,y)\to\infty$ if $(\vx,y)\in\cF$ and $\|(\vx,y)\|\to\infty$;
\end{enumerate}
If the feasible set of $(\vx,y)$ is bounded, then \ref{A_coercive} holds trivially for any continuous objective function. Therefore, \ref{A_coercive} is much weaker than assuming that the objective function is coercive over the entire space $\RR^{n+q}$. The assumption \ref{A_coercive} can be dropped if the boundedness of the sequence can be deducted from other means.

Within the proof, $A_ix_i^k$ and $By^k$ often appear in the first order conditions (e.g. see equations \eqref{fsubg}, \eqref{condpi1}). In order to have a reverse control, i.e., controlling $x_i^k, y^k$ based on $A_ix_i^k, By^k$, we need the following two assumptions on matrices $A_i$ and $B$.
\begin{enumerate}[label={A\arabic*}]
\addtocounter{enumi}{1}
\item\label{A_feasible} \textbf{(feasibility)} $\image{\vA} \subseteq \image{B}$, where $\image{\cdot}$ returns the image of a matrix;
\end{enumerate}
\begin{enumerate}[label={A\arabic*}]
\addtocounter{enumi}{2}
\item\label{A_rank} \textbf{(Lipschitz sub-minimization paths)}
\begin{enumerate}
\item For any fixed $\vx$, $\argmin_y\{ \phi(\vx,y): By = u\}$ has a unique minimizer. $H:\mathrm{Im}(B)\rightarrow \mathbb{R}^{q}$ defined by $H(u)\triangleq \argmin_y\{ \phi(\vx,y): By = u\}$ is a Lipschitz continuous map.
\item For $i=0,\ldots,p$ and any $x_{<i}$, $x_{>i}$ and $y$, $\argmin_{x_i}\{ \phi(x_{<i},x_i,x_{>i},y) :A_ix_i = u\}$ has a unique minimizer and $F_i : \mathrm{Im}(A_i)\rightarrow \mathbb{R}^{n_i}$ defined by $F_i(u) \triangleq \argmin_{x_i}\{ \phi(x_{<i},x_i,x_{>i},y) :A_ix_i = u\}$ is a Lipschitz continuous map.
\end{enumerate}
Moreover, the above $F_i$ and $H$ have a universal Lipschitz constant $\bar{M}>0$.
\end{enumerate}
These two assumptions allow us to control $x^k_i, y^k$ by $A_ix_i^k, By^k$ as in Lemma \ref{lemma:reverse_control}.

\begin{lemma}\label{lemma:reverse_control}
It holds that, $\forall k_1,k_2\in \mathbb{N}$,
\begin{align}
\|y^{k_1} - y^{k_2}\| &\leq \bar{M} \|By^{k_1} - By^{k_2}\|,\label{MBy}\\
\|x_i^{k_1} - x_i^{k_2}\|&\leq\bar{M} \|A_ix_i^{k_1} - A_ix_i^{k_2}\|,\quad i=0, 1, \ldots,p,\label{MAx}
\end{align}
where $\bar{M}$ is given in \ref{A_rank}.
\end{lemma}

They weaken the full column rank assumption typically imposed on matrices $A_i$ and $B$. When $A_i$ and $B$ have full column rank,  their null spaces are trivial and, therefore, $F_i,H$ reduce to linear operators and satisfy \ref{A_rank}. However, the assumption \ref{A_rank} allows non-trivial null spaces and holds for more functions. For example, if a function $f$ is a $C^2$ with its Hessian matrix $H$ bounded everywhere $\sigma_1 I \succeq H \succeq \sigma_2 I$ ($\sigma_1 > \sigma_2 > 0$), then $F$ satisfies \ref{A_rank} for any matrix $A$.
If the uniqueness fails to hold, i.e., there exists $y_1,y_2$ such that $By_1 = By_2$ and $\phi(\vx, y_1) = \phi(\vx, y_2)$, then the augmented Lagrangian cannot distinguish them, causing troubles to the boundedness of the sequence.

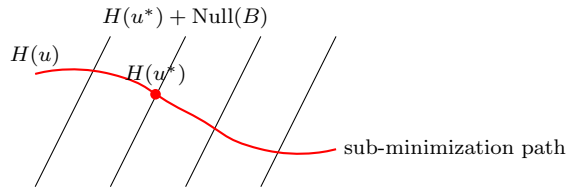
\begin{figure}[ht]
\centering
\begin{tikzpicture}[scale=1]
    \draw (0,0) coordinate (a_1) -- (1,2) coordinate (a_2);
    \draw (1,0) coordinate (a_3) -- (2,2) coordinate (a_4);
    \draw (2,0) coordinate (a_5) -- (3,2) coordinate (a_6);
    \draw (3,0) coordinate (a_5) -- (4,2) coordinate (a_6);
    \draw [red,thick] plot [smooth, tension=1] coordinates{ (0,1.5) (1,1.5) (2,1) (3,0.5) (4,0.5)};
    \draw (a_4) node [above] {$H(u^*) + \mathrm{Null}(B)$};
    \coordinate (xx) at (1.6,1.23);
    \draw[thick] (xx) node [above] {$H(u^*)$};
    \fill[red] (xx) circle (2pt);
    \draw (0,1.5) node [above] {$H(u) $};
    \draw (4,0.5) node [right] {sub-minimization path};
\end{tikzpicture}
\centering
\caption{Illustration of the assumption \ref{A_rank}, which assumes that  $H(u) = \argmin\{h(y): By = u\}$ is Lipschitz  \cite{Rosenberg1988}.}\label{fig:diagram_infimal}
\end{figure}

As for the objective function, we consider two different scenarios:
\begin{itemize}
  \item Theorem \ref{thm:main} considers the scenario where $\vx$ and $y$ are decoupled in the objective function;
  \item Theorem \ref{thm:main_theorem_convex} considers the scenario where $\vx$ and $y$ are possibly coupled but their function $\phi(\vx,y)$ is Lipschitz differentiable.
\end{itemize}
The model in the first scenario is
\begin{align}\label{Eq:originalForm1}
  \Min_{x_0,x_1,\ldots,x_p,y}\quad & \hspace{12pt}f(x_0,x_1,\ldots,x_p)\hspace{12pt} + h(y)\\\nonumber
  \St\quad & A_0x_0 +A_1x_1 +\cdots+A_px_p + B y = b,
\end{align}
where the function $f:\RR^n\to \RR\cup\{\infty\}$ ($n = \sum_{i=0}^p n_i$) is proper, continuous, and possibly nonsmooth, and the function  $h:\RR^q\to\RR$ is proper and differentiable. Both $f$ and $h$ can be nonconvex.
\begin{theorem}\label{thm:main}
Suppose that \ref{A_coercive}-\ref{A_rank} and the following assumptions hold.
\begin{enumerate}[label={A\arabic*}]
\addtocounter{enumi}{3}
\item\label{A_f} \textbf{(objective-$f$ regularity)}  $f$ has the form
$$f(\vx):=g(\vx)+ \sum_{i=0}^p f_i(x_i) $$
where
\begin{enumerate}
\item[(i)]$g(\vx)$ is Lipschitz differentiable with constant $L_g$,
\item[(ii)]Either
\begin{enumerate}
\item[a.] $f_0$ is lower semi-continuous, $f_i(x_i)$ is restricted prox-regular (Definition \ref{def:quadcvx}) for $i=1,\ldots, p$; Or,
\item[b.] The supremum $\sup\{\|d\|: x_0\in S, d\in \partial f_0(x_0)\} $ is bounded for any bounded set $S$, $f_i(x_i)$ is continuous and piecewise linear (Definition \ref{def:piecewise_linear}) for $i=1,\ldots, p$;
\end{enumerate}

\end{enumerate}
\item\label{A_h} \textbf{(objective-$h$ regularity)} $h(y)$ is Lipschitz differentiable with constant $L_h$;
\end{enumerate}
Then, Algorithm \ref{Alg:BCD} converges subsequently for any sufficiently large $\beta$ (the lower bound is given in Lemma \ref{lm:p2}), that is, starting from any $x_1^0,\ldots,x_p^0,y^0,w^0$, it generates a sequence that is bounded, has at least one limit point, and that each limit point $(\vx^*,y^*,w^*)$ is a stationary point of $\Lag_\beta$, namely, $0\in\partial \Lag_\beta(\vx^*,y^*,w^*)$.

In addition, if $\Lag_\beta$ is a Kurdyka-{\L}ojasiewicz (K{\L}) function
 \cite{lojasiewicz1993geometrie,bolte2007lojasiewicz,attouch2011}, then $(\vx^k,y^k,w^k)$ converges globally\footnote{ \textit{"Globally"} here means regardless of where the initial point is.} to the unique limit point $(\vx^*,y^*,w^*)$.
\end{theorem}

Assumptions \ref{A_f} and \ref{A_h}  regulate the objective functions. None of the functions needs to be convex. $f_0$ can be any lower semi-continuous function, and the non-Lipschitz differentiable parts $f_1,\ldots,f_n$ of $f$  shall satisfy either Definition \ref{def:piecewise_linear} or Definition \ref{def:quadcvx}. Under Assumptions \ref{A_f} and \ref{A_h}, the augmented Lagrangian function $\Lag_{\beta}$ is lower semi-continuous.

It will be easy to see, from our proof in Section 3.3, that the Lipschitz differentiable assumption on $g$ can be relaxed to hold just in any bounded set, since the boundedness of $\{\vx^k\}$ is established before that property is used in our proof. Consequently, $g$ can be functions like $e^x$, whose derivative is not globally Lipschitz.


Functions satisfying the K{\L} inequality include real analytic functions, semi-algebraic functions and locally strongly convex functions (more information can be referred to Sec. 2.2 in {\cite{xu2014a}} and references therein).

In the second scenario, $\vx$ and $y$ can be coupled in the objective as shown in \eqref{Eq:originalForm0}, but the objective needs to be smooth.
\begin{theorem}\label{thm:main_theorem_convex}
Suppose that \ref{A_coercive}-\ref{A_rank} hold and $\phi$ in \eqref{Eq:originalForm0} is Lipschitz differentiable with constant $L_\phi$. Then, Algorithm \ref{Alg:BCD} has the same subsequential and global convergence results as stated in Theorem \ref{thm:main}.
%
\end{theorem}
Although Theorems \ref{thm:main} and \ref{thm:main_theorem_convex} impose different conditions on the objective functions, their proofs are similar. Hence, we will focus on proving Theorem \ref{thm:main} first and leave the proof of Theorem \ref{thm:main_theorem_convex} to the Appendix.
\section{Proof}\label{sec:proof}
\subsection{Keystones}
The following properties hold for Algorithm \ref{Alg:BCD} under our assumptions. Here, we first list them and present Proposition \ref{prop:convergence}, which  establishes convergence assuming these properties. Then in the next two subsections, we prove these properties.
\begin{enumerate}[label={P\arabic*}]
\item\label{P1} \textbf{(Boundedness)} $\{\vx^k, y^k, w^k\}$ is bounded, and $\Lag_\beta(\vx^k,y^k,w^k)$ is lower bounded.
\item\label{P2} \textbf{(Sufficient descent)} There is a constant $C_1(\beta) >0$ such that for all sufficiently large $k$, we have
\begin{align}\label{Cond:1}
&\Lag_\beta(\vx^{k},y^k,w^{k}) - \Lag_\beta(\vx^{k+1},y^{k+1},w^{k+1}) \geq C_1(\beta) \Big(\|B (y^{k+1} - y^k)\|^2 +\sum_{i=1}^p \|A_i(x_i^{k}-x_i^{k+1})\|^2\Big).
\end{align}
\item\label{P3} \textbf{(Subgradient bound)} There exists $C_2(\beta) > 0$ and $d^{k+1}\in \partial \Lag_\beta(\vx^{k+1},y^{k+1},w^{k+1})$ such that
\begin{align}\label{Cond:2}
\|d^{k+1}\| &\leq C_2(\beta) \Big(\|B(y^{k+1}-y^k)\|+\sum_{i=1}^p\|A_i(x_i^{k+1}-x_i^k)\|\Big).
\end{align}
It is our intention to start $i$ at 1, thus skipping the $x_0$-block, in \eqref{Cond:1} and \eqref{Cond:2}.
\item\label{P4} \textbf{(Limiting continuity)} If $(\vx^*, y^*, w^*)$ is the limit point of a sub-sequence $(\vx^{k_s}, y^{k_s}, w^{k_s})$ for $s \in \mathbb N$, then $\Lag_\beta(\vx^*, y^*, w^*) = \lim_{s\rightarrow \infty} \Lag_\beta(\vx^{k_s}, y^{k_s}, w^{k_s})$.
\end{enumerate}
The proposition below is standard and not new though it does not appear exactly in the literature.
\begin{proposition}\label{prop:convergence} Suppose that when an algorithm is applied to the problem \eqref{Eq:originalForm1}, its sequence $(\vx^k,y^k,w^k)$ satisfies \ref{P1}--\ref{P4}. Then, the sequence has at least a limit point $(\vx^*,y^*,w^*)$, and any limit point $(\vx^*,y^*,w^*)$ is a stationary point. That is, $0\in \partial \Lag_\beta(\vx^*,y^*,w^*)$, or equivalently,
\begin{subequations}
\begin{align}
0&=\vA\vx^*+By^*,\\
0&\in \partial f(\vx^*) + \vA^T w^*,\\
0&\in \partial h(y^*) + B^T w^*.
\end{align}
\end{subequations}
Furthermore, the running best rates\footnote{A nonnegative sequence ${a_k}$ induces its running best sequence $b_k=\min\{a_i : i\le k\}$; therefore, ${a_k}$ has running best rate of $o(1/k)$ if $b_k=o(1/k)$.} of the sequences $\{\|B (y^{k+1} - y^k)\|^2 +\sum_{i=1}^p \|A_i(x_i^{k}-x_i^{k+1})\|^2\}$  and  $\{\|d^{k+1}\|\}$ are $o(\frac{1}{k})$ and $o(\frac{1}{\sqrt{k}})$, respectively.
Moreover, if $\Lag_\beta$ is a K{\L} function, then $(\vx^k,y^k,w^k)$ converges globally to the unique point $(\vx^*,y^*,w^*)$.
\end{proposition}
\begin{proof} The proof is standard. Similar steps are found in, for example, \cite{attouch2011,xu2014a}.

By \ref{P1}, the sequence $(\vx^k,y^k,w^k)$ is bounded, so there exist a convergent subsequence and a limit point, denoted by  $(\vx^{k_s},y^{k_s},w^{k_s})_{s\in \mathbb{N}} \rightarrow (\vx^*,y^*,w^*)$ as $s\rightarrow +\infty$. By \ref{P1} and \ref{P2}, $\Lag_\beta(\vx^{k},y^k,w^{k})$ is monotonically nonincreasing and lower bounded, and therefore $\|A_i x_i^{k} - A_i x_i^{k+1}\|\rightarrow 0$ and $\|B y^{k} - By^{k+1}\|\rightarrow 0$ as $k\rightarrow \infty$. Based on \ref{P3},  there exists $d^{k}\in \partial \Lag_{\beta}(\vx^{k},y^{k},w^{k})$ such that $\|d^{k}\|\rightarrow 0$.
In particular, $\|d^{k_s}\|\rightarrow 0$ as $s\to\infty$. Based on \ref{P4}, $\Lag_\beta(\vx^*, y^*, w^*) = \lim_{s} \Lag_\beta(\vx^{k_s}, y^{k_s}, w^{k_s})$. By definition of general subgradient\cut{ (see Definition \ref{def:general_sub_gradient} or} \cite[Definition 8.3]{VariationalAnalysis-RockWets}, we have $0\in \partial \Lag_{\beta}(\vx^*,y^*,w^*).$

The running best rate of the sequence $\{\|B (y^{k+1} - y^k)\|^2 +\sum_{i=1}^p \|A_i(x_i^{k}-x_i^{k+1})\|^2\}$ can be easily obtained via taking advantage of \cite[Lemma 1.2]{Deng2013} or \cite[Theorem 3.3.1]{Knopp1956}. By \eqref{Cond:2}, it is obvious that the running best rate of the sequence $\{\|d^{k+1}\|\}$ is $o(\frac{1}{\sqrt{k}})$.

Similar to the proof of Theorem 2.9 in \cite{attouch2011}, we can claim the global convergence of the considered sequence $(\vx^k,y^k,w^k)_{k\in \mathbb{N}}$ under the K{\L} assumption of $\Lag_{\beta}$.
\qed
\end{proof}
In \ref{P2}, the sufficient descent inequality \eqref{Cond:1} is only required for any sufficiently large $k$, not all $k$. In our analysis, \ref{P1} gives subsequence convergence, \ref{P2} measures the augmented Lagrangian descent, and  \ref{P3} bounds the subgradient by total point changes. The reader may still obtain \ref{P1}--\ref{P4} when generalizing Algorithm \ref{Alg:BCD}, for example, by replacing the direct minimization subproblems to prox-gradient or inexact subproblems and by relaxing the ordering in which the primal variables are updated.

\cut{Our results remain valid when the subproblems in Algorithm \ref{Alg:BCD} are replaced by
\begin{align*}
x_i^{k+1} \gets \argmin_{x_i} \Lag_\beta(x^{k+1}_{<i},x_i,x_{>i}^k,y^k,w^k) + C(x_i - x^k_i) A_i^TA_i(x_i - x_i^k),
\end{align*}
where $C>0$, }
\subsection{Preliminaries}

In this subsection, we give some useful lemmas that will be used in the main proof.
To save space, throughout this section we assume assumptions \ref{A_coercive}--\ref{A_h} hold, and let
\beq\label{itrsimple}
(\vx^+,y^+,w^+):=(\vx^{k+1},y^{k+1},w^{k+1}).
\eeq
In addition, we let $A_{<s}x_{<s} := \sum_{i<s} A_i x_i$ and, in a similar fashion, $A_{>s}x_{>s}:= \sum_{i>s}A_i x_i$.

\begin{lemma}\label{lemma:well_define}
If $\beta>\bar{M}^2L_h$ ($\bar{M}$ is defined in \ref{A_rank}), all the subproblems in Algorithm \ref{Alg:BCD} are well defined.
\end{lemma}
This lemma is on its own, so we leave its proof to the appendix.

\begin{lemma}[bound dual by primal]\label{lemma:w_derivative}
\cut{Define} Let $\lambda_{++}(B^TB)$ be the smallest strictly-positive eigenvalue of $B^TB$, $C \triangleq L_h \bar{M}\lambda_{++}^{-1/2}(B^TB)$. For all $k\in \mathbb{N}$, it holds that
\begin{enumerate}
\item[(a)]$ B^Tw^k = - \nabla h(y^k)$.
\item[(b)] \label{lemma:y_control_w}$\|w^+ - w^k\| \leq C \|By^+ - By^k\|.$
\end{enumerate}

\end{lemma}
\begin{proof} Part (a) follows directly from the optimality condition of $y^k$:
$0 = \nabla h(y^k) + B^T w^{k-1} + \beta B^T (A\vx^k + By^k),$
and $w^{k} = w^{k-1} + \beta\left(\vA\vx^{k}+By^{k}\right)$.

Then let us prove Part (b).
Since $w^+ - w^k = \beta (A\vx^+ + By^+) \in \mathrm{Im}(B)$, we get \[\|w^+ - w^k \| \leq \lambda_{++}^{-1/2}(B^TB) \|B^T(w^+ - w^k)\|= \lambda_{++}^{-1/2}(B^TB)\|\nabla h(y^+) - \nabla h(y^k) \|\le C \|By^+ - By^k\|.\]
The last inequality follows from the Lipschitz property of $\nabla h$ and Lemma \ref{lemma:reverse_control}.
\qed
\end{proof}

\subsection{Main proof}
This subsection proves Theorem \ref{thm:main}  for Algorithm \ref{Alg:BCD} under {Assumptions \ref{A_coercive}--\ref{A_h}}.
For all $k\in\mathbb{N}$ and $i=0,\ldots,p$,  because of the optimality of $x_i^k$,
we can introduce the following \emph{general subgradients} $d_i^k$ and $\bar{d}_i^k$,
\begin{align}
\label{fsubg}
\bar{d}_i^k&:=-(A_i^Tw^+ +\beta \rho_{i}^k)\in \partial_i  f(x^+_{< i},x^+_i,x^k_{>i}),\\
\label{condpi1}
d_i^k&:=-\nabla_{i}g(x^+_{<i},x^+_i,x^{k}_{>i})+\bar{d}_i^k \in \partial f_i(x^+_i),
\end{align}
where $$\rho_{i}^k:= A_i^T(A_{>i}x_{>i}^{k}-A_{>i}x_{>i}^+)+A_i^T(By^{k}-By^+).$$

The next two lemmas estimate the descent of $\Lag_\beta(\vx,y,w)$ at each iteration.
\begin{lemma}[descent of $\Lag_\beta$ during $x_i$ update]\label{lm:x} The iterates in Algorithm \ref{Alg:BCD} satisfy
\begin{enumerate}
\item $\Lag_\beta(x^+_{< i},{\boldsymbol x^k_i},x^k_{>i},y^k,w^k)\ge \Lag_\beta(x^+_{< i},{\boldsymbol x^+_i},x^k_{>i},y^k,w^k)$, $i=0,\ldots,p$;\\[-4pt]
\item $\Lag_\beta(\vx^k,y^k,w^k)\ge \Lag_\beta(\vx^+,y^k,w^k)$;\\[-4pt]
\item $\Lag_\beta(\vx^k,y^k,w^k)-\Lag_\beta(\vx^+,y^k,w^k) =\sum_{i=0}^pr_i$, where
\begin{align}
r_i&:=f(x^+_{< i},x^k_i,x^k_{>i})-
f(x^+_{< i},x^+_i,x^k_{>i})
-\dotp{\bar{d}_i^k,x^k_i-x^{+}_i}+\frac{\beta}{2}\|A_ix^k_i-A_i x^+_i\|^2\geq 0,
\label{formri}
\end{align}
where $\bar{d}_i^{k}$ is defined in \eqref{fsubg}.
\item For $i = 1,\ldots, p$ (without the block $i = 0$), if
\beq\label{quadcvx1}
f_i(x_i^k)+\frac{\gamma_i}{2}\|x_i^k-x_i^+\|^2\ge f_i(x_i^+)+\dotp{d_i^k,x_i^k-x_i^+},
\eeq
holds with constant $\gamma_i\geq 0$ (later, this condition will be shown to hold), then we have
\beq \label{rilb}
r_i\ge \frac{\beta-\gamma_i  \bar{M}^2-L_g \bar{M}^2}{2}\|A_ix_i^k-A_ix_i^+\|^2,
\eeq
where the constants $L_g$ and $\bar{M}$  are defined in Assumptions \ref{A_f} and \ref{A_rank}, respectively.
\end{enumerate}
\end{lemma}
\begin{proof}
\textbf{Part 1} follows directly from the minimization subproblems, which give $x^+_i$. \textbf{Part 2} is a result of
\begin{align*}
\Lag_\beta(\vx^k,y^k,w^k)-\Lag_\beta(\vx^+,y^k,w^k)& =\sum_{i=0}^p\big(\Lag_\beta(x^+_{< i},x^k_i,x^k_{>i},y^k,w^k)-\Lag_\beta(x^+_{< i},x^+_i,x^k_{>i},y^k,w^k) \big),
\end{align*}
and part 1.
\textbf{Part 3}: Each term in the sum  equals  $f(x^+_{< i},x^k_i,x^k_{>i})-
f(x^+_{< i},x^+_i,x^k_{>i})$ \emph{plus}
\begin{align*}
&\dotp{w^k,A_ix^k_i-A_ix^+_i}+\frac{\beta}{2}\|A_{<i}x^+_{< i}+A_ix^k_i+A_{>i}x^k_{>i}+By^k\|^2
-\frac{\beta}{2}\|A_{<i}x^+_{< i}+A_ix^+_i+A_{>i}x^k_{>i}+By^k\|^2&\\
&=\dotp{w^k,A_ix^k_i-A_ix^+_i}+\dotp{\beta\left(A_{<i}x^+_{<i}+A_ix^+_i+A_{> i}x^k_{>i} + By^k\right),A_ix^k_i-A_ix^{+}_i}+\frac{\beta}{2}\|A_ix^k_i-A_i x^+_i\|^2
\\
&=\dotp{A_i^Tw^+ +\beta \rho_i^k,x^k_i-x^{+}_i}+\frac{\beta}{2}\|A_ix^k_i-A_i x^+_i\|^2
\end{align*}
where the first equality follows from the cosine rule: $\|b+c\|^2-\|a+c\|^2=\|b-a\|^2+2\dotp{a+c,b-a}$ with $b=A_ix^k_i$, $a=A_ix^+_i$, and $c=A_{<i}x^+_{< i}+A_{>i}x^k_{>i}+By^k$.
\smallskip

\noindent\textbf{Part 4.} Let $d_i^{k}$ be defined in \eqref{condpi1}. From the inequalities \eqref{MAx} and \eqref{quadcvx1}, we get
\begin{align}\label{Eq:find_gamma_i}
f_i(x^k_i)- f_i(x^+_i)-\dotp{d_i^k,x^k_i-x^+_i} \ge-\frac{\gamma_i}{2}\|x^k_i-x^+_i\|^2 \ge -\frac{\gamma_i \bar{M}^2}{2}\|Ax^k_i-Ax^+_i\|^2.
\end{align}
By the assumption \ref{A_f} part (i) and \eqref{MAx}, we also get
\begin{align}\label{Eq:find_Li}
g(x^+_{<i},x^k_i,x^k_{>i})-g(x^+_{< i},{x^+_i},x^k_{>i})-\dotp{\nabla_{i}g(x^+_{< i},{x^+_i},x^k_{>i}), x^k_i-x^{+}_i}\ge-\frac{L_{g}}{2}\|x^k_i-x^+_i\|^2 \ge -\frac{L_{g} \bar{M}^2}{2}\|Ax^k_i-Ax^+_i\|^2.
\end{align}
Finally,  rewriting the expression of $r_i$ and applying \eqref{Eq:find_gamma_i} and \eqref{Eq:find_Li} we obtain
\begin{align*}
r_i&= \big(g(x^+_{<i},x^k_i,x^k_{>i})-g(x^+_{< i},{x^+_i},x^k_{>i})-\dotp{\nabla_{i}g(x^+_{< i},{x^+_i},x^k_{>i}), x^k_i-x^{+}_i}\big)\\
&\quad+\big(f_i(x^k_i)- f_i(x^+_i)-\dotp{d_i^k,x^k_i-x^+_i}\big) +\frac{\beta}{2}\|Ax^k_i-Ax^+_i\|^2\\
&\ge \frac{\beta-\gamma_i  \bar{M}^2-L_g \bar{M}^2}{2}\|A_ix_i^k-A_ix_i^+\|^2.
\end{align*}
\qed
\end{proof}
The assumption \eqref{quadcvx1} in the part 4 of Lemma \ref{lm:x} is the same as \eqref{quadcvx} in Definition \ref{def:quadcvx} except the latter holds for more functions due to the exclusion set $S_M$.
In order to relax \eqref{quadcvx1} to \eqref{quadcvx}, we must find $M$ and specify the exclusion set $S_M$. (This complicates our analysis but is necessary for many nonconvex functions such as the $\ell_q$ quasi-norm.) We will finally achieve this relaxation in Lemma \ref{lm:p2}.

\begin{lemma}[descent of $\Lag_\beta$ due to $y$ and $w$ updates]\label{lm:yw} If $\beta>2(L_h\bar{M}^2 + 1 + C)$, where $C$ is the constant specified in Lemma \ref{lemma:y_control_w} and $L_h$ is the Lipschitz constant in Assumption \ref{A_h}, then for any $k\in\mathbb{N}$
\begin{align}\label{Eq:by_diff}
\Lag_\beta(\vx^+,y^k,w^k)-\Lag_\beta(\vx^+,y^+,w^+) &\geq \|B y^+-B y^k\|^2.
\end{align}
\end{lemma}
\begin{proof}
Because $\beta/2 > L_h\bar{M}^2 + 1 + C$ and $\beta^{-1} < 1/C$, we know
\begin{equation}\label{Eq:obviousBound}
\frac{\beta}{2} - \frac{C^2}{\beta} - \frac{L_h\bar{M}^2}{2} > L_h\bar{M}^2 + 1 + C - C - \frac{L_h\bar{M}^2}{2} > 1.
\end{equation}

From the assumption \ref{A_h} and  Lemma \ref{lemma:w_derivative}(b), it follows
\begin{align}\nonumber
& \quad \Lag_{\beta}(\vx^+,y^k,w^k)-\Lag_{\beta}(\vx^+,y^+,w^+) \\
& = h(y^k)-h(y^+)+\dotp{w^+,By^k-By^+}+\frac{\beta}{2}\|B y^+-B y^k\|^2-\frac{1}{\beta}\|w^+-w^k\|^2\\ \label{Eq:beta_lowerbound_2}
& \geq -\frac{L_h\bar{M}^2}{2}\|B y^+-B y^k\|^2+\frac{\beta}{2}\|B y^+-B y^k\|^2-\frac{C^2}{\beta }\|B y^+-B y^k\|^2\\
& \geq \|B y^+-B y^k\|^2, \nonumber
\end{align}
The first inequality holds because
\begin{align*}
&h(y^k) - h(y^+)  + \big<w^+, By^k - By^+\big> \\
= &h(y^k) - h(y^+)  + \big<B^Tw^+, y^k - y^+\big> \\
= &h(y^k) - h(y^+)  - \big<\nabla h(y^+), y^k - y^+\big> \\
= & -\frac{L_h}{2}\|y^k - y^+\|^2 \quad \text{(Lipschitz differentiable of $-h$)}\\
= & -\frac{L_h\bar M}{2}\|By^k - By^+\|^2.
\end{align*}
The last inequality holds because of \eqref{Eq:obviousBound}. \qed
\end{proof}
Based on Lemma \ref{lm:x} and Lemma \ref{lm:yw}, we now establish the following results:
\begin{lemma}[Monotone, lower--bounded $\Lag_\beta$ and (\ref{P1}) bounded sequence]\label{lm:mono} If $\beta > 2(L_h\bar{M}^2 + 1 + C)$ as in Lemma \ref{lm:yw}, then the sequence $(\vx^k,y^k,w^k)$ generated by Algorithm \ref{Alg:BCD} satisfies
\begin{enumerate}
\item $\Lag_\beta(\vx^k,y^k,w^k)\ge \Lag_\beta(\vx^+,y^+,w^+)$.
\item $\Lag_\beta(\vx^k,y^k,w^k)$ is lower bounded for all $k\in \mathbb{N}$ and converges as $k\to\infty$.
\item $\{\vx^k,y^k,w^k\}$ is bounded.
\end{enumerate}
\end{lemma}
\begin{proof}
Part 1. It is a direct result of Lemma \ref{lm:x} part 2, and Lemma \ref{lm:yw}.

Part 2. By the assumption \ref{A_feasible}, there exists $y'$ such that $\vA\vx^k+By'=0$ and $y' = H(By')$. By the assumptions \ref{A_coercive}--\ref{A_feasible}, we have
$$f(\vx^k) + h(y')\geq  \min_{\vx,y} \{f(\vx)+h(y):\vA\vx+By=0\}>-\infty.$$Then we have
\begin{align*}
\Lag_\beta(\vx^k,y^k,w^k) &= f(\vx^k) + h(y^k) + \dotp{B^Tw^k,y^k-y'}+ \frac{\beta}{2}\|\vA\vx^k+By ^k\|^2\\
&= f(\vx^k) + h(y^k) +\dotp{\nabla h(y^k),y'-y^k}+ \frac{\beta}{2}\|\vA\vx^k+By^k \|^2\\
(\mbox{Lemma~}\ref{lemma:reverse_control}, \nabla h~\mbox{is Lipschitz}) \quad&\ge f(\vx^k) + h(y') + \frac{\beta-L_h\bar{M}^2}{2}\|\vA\vx^k+By ^k\|^2\\
&>-\infty.
\end{align*}

Part 3. From parts 1 and 2, $\Lag_\beta(\vx^k,y^k,w^k)$ is upper bounded by $\Lag_\beta(\vx^0,y^0,w^0)$ and so are $f(\vx^k) + h(y') $ and $\|\vA\vx^k+By^k \|^2$. By the assumption \ref{A_coercive}, $\{\vx^k\}$ is bounded and, therefore, $\{By^k\}$ is also bounded. By Lemma \ref{lemma:reverse_control}, we know that $\{y^k\}$  is bounded. By Lemma \ref{lemma:w_derivative}, $\{B^Tw^k\}$ is also bounded. Similar to the proof in Lemma \ref{lemma:w_derivative}(b), $w^k - w^0\in \mathrm{Im}(B)$. Therefore, the boundedness of $B^Tw^k$ implies the boundedness of $w^k$.\qed
\end{proof}
It is important to note that, once $\beta$ is larger than the threshold, the constants and bounds in Lemmas \ref{lm:yw} and \ref{lm:mono} only rely on the objective $f(x) + h(y)$, matrices $\vA$, $B$, and the initial point $(\vx^0, y^0, w^0)$ but will be \emph{independent of} $\beta$, which is essential to the proof of Lemma \ref{lm:p2} below.
\begin{lemma}[Asymptotic regularity]\label{lm:By} $\lim_{k\to\infty}\|By^k-By^{+}\|=0$ and $\lim_{k\to\infty}\|w^k-w^{+}\|=0$.
\end{lemma}
\begin{proof} The first result follows directly from Lemmas \ref{lm:x}, \ref{lm:yw}, and \ref{lm:mono} (part 2), and the second result from Lemma \ref{lemma:w_derivative} part (b).
\end{proof}
The lemma below corresponds to the assumption A4, part(ii)-b.
\begin{lemma}[Boundedness for piecewise linear $f_i$'s]\label{lm:pl} Consider the case that $f_i$,  $i=1,\ldots,p$,  are piece-wise linear. There exist constants $M^* > 0$ (independent of $\beta$), $\bar{M}$ and $L_g$ defined in \ref{A_rank} and \ref{A_f}, respectively, for any $\epsilon_0>0$, when $\beta >  \max\{2(M^*+1)/{\epsilon_0^2},L_h\bar{M}^2 + 1 + C\}$, there exists $k_{\mathrm{pl}}\in\mathbb{N}$ such that the followings hold for all $k>k_{\mathrm{pl}}$:
\begin{enumerate}
\item $\|A_ix^+_i - A_ix^k_i\| < \epsilon_0$ and $\|x^+_i - x^k_i\| < \bar{M}\epsilon_0$, $i=0,\ldots,p$;\\[-8pt]
\item $\|\nabla g(\vx^k) - \nabla g(\vx^+)\| < (p+1)\bar{M}L_g\epsilon_0$.
\end{enumerate}
\end{lemma}
\begin{proof}
Part 1. Since the number $K$ of the linear pieces of $f_i$ is finite for $i=1,\ldots, p$, $\partial f_0$ is bounded for $x$ in any bounded set $S$, and $\{\vx^k,y^k,w^k\}$ is bounded (see Lemma \ref{lm:mono}),
 $\partial_i  f(x^+_{< i},x^+_i,x^k_{>i})$ are uniformly bounded for all $k$ and $i$. Since $\bar{d}_i^k\in \partial_i  f(x^+_{< i},x^+_i,x^k_{>i})$ (see \eqref{fsubg}), the first three terms of $r_i$ (see \eqref{formri}) are bounded by a universal constant $M^*$ independent of $\beta$:
$$f(x^+_{< i},x^k_i,x^k_{>i})-f(x^+_{< i},x^+_i,x^k_{>i})-\dotp{\bar{d}_i^k,x^k_i-x^{+}_i}\in [-M^*,M^*].$$
Hence, as long as $\beta>2(M^*+1)/{\epsilon_0^2}$,
\begin{align}\label{Eq:find_beta_0}
\|A_ix^+_i - A_ix^k_i\| \ge \epsilon_0 &\Rightarrow ~ r_i\ge \frac{\beta}{2} \epsilon_0^2-M^* > 1\\
 &\Rightarrow ~\Lag_\beta(x^+_{< i},{\boldsymbol x^k_i},x^k_{>i},y^k,w^k)- 1> \Lag_\beta(x^+_{< i},{\boldsymbol x^+_i},x^k_{>i},y^k,w^k).
\end{align}
By Lemmas \ref{lm:x}, \ref{lm:yw}, and \ref{lm:mono}, this means $\Lag_\beta(\vx^k,y^k,w^k)-1>\Lag_\beta(\vx^+,y^+,w^+)$. Since $\{\Lag_\beta(\vx^k,y^k,w^k)\}$ is lower bounded, $\|A_ix^+_i - A_ix^k_i\| \ge \epsilon_0$ can only hold for finitely many $k$.  Thus for $i = 1,\ldots, p$, we have
\[
\|A_ix^+_i - A_ix^k_i\| < \epsilon_0.
\]
As for $i = 0$, because of Lemma \ref{lm:By}, we know
\[
\limsup_k \|A_0x^+_0 - A_0x^k_0\| \leq \limsup_k \|\sum_{i=1}^p (A_ix^+_i - A_ix^k_i) + By^+ - By^k\| \leq p\epsilon_0.
\]
Thus for large $k > k_{pl}$, $\|A_0 x^+_0 - A_0 x^k_0\| \leq (p+1) \epsilon_0$
By Lemma \ref{lemma:reverse_control}, we know Part 1 is correct.

Part 2 follows from $\|\nabla g(\vx^k) - \nabla g(\vx^+)\|\le L_g\|\vx^k-\vx^+\|$, part 1 above, and Lemma \ref{lemma:reverse_control}.
\qed
\end{proof}

\begin{lemma}[Sufficient descent property  \ref{P2}]\label{lm:p2} Suppose
$$\beta>\max\Big\{2(M+1)/{\epsilon_0^2},L_h\bar{M}^2 + 1 + C,\sum_{i=1}^p\gamma_i\bar{M}^2 + L_g\bar{M}^2\Big\},$$
where $\gamma_i\ (i=1,\ldots,p)$ and $\epsilon_0$ are constants only depending on $f$, $M >  M^*$ is a constant independent of $\beta$. Then, Algorithm \ref{Alg:BCD} satisfies the sufficient descent property \ref{P2}.
\end{lemma}
It is worth noting that the proof below will be much simpler if there are only two blocks, instead of $p+2$, or if we assume \emph{prox-regular} functions $f_i$ instead of the less restrictive \emph{restricted prox-regular} functions.
\begin{proof}
\cut{By substituting $w^+=w^k + \beta\left(\vA\vx^{k+1}+By^{k+1}\right)$  and introducing
$$\rho_{i}= A_i^T\sum_{j=i+1}^p(A_jx_j^k-A_jx_j^+)-A_p^T(By^k-By^+),$$
the optimality condition  $0\in\partial f_i(x^+_i)+\nabla_{i}g(x^+_{<i},x^+_i,x^k_{>i}) + A_i^T\big(w^k+\beta(A_{<i}x^+_{<i}+A_ix^+_i+A_{>i}x^k_{>i}+By^k)\big)$ can be rewritten as

\beq\label{condpi1}
-(\nabla_{i}g(x^+_{<i},x^+_i,x^k_{>i})+A_p^Tw^+)-\beta \rho_i \in \partial f_i(x^+_i). \eeq
 By Lemma \ref{lm:mono} part 3, there exists a universal $M>0$ (independent of $\beta,k,i$) such that
$$\big\|\nabla_{i}g(\vx^+)+A_i^Tw^+\big\|\le (M-1),\quad i=1,\ldots,p.$$}

We will show the lower bound \eqref{rilb} for  $i=1,\ldots,p$, which, along with Lemma \ref{lm:x} part 3 and Lemma \ref{lm:yw}, establishes the sufficient descent property \ref{P2}.

We shall obtain the lower bound \eqref{rilb}  in the backward order $i=p,(p-1),\ldots,1$.  In light of Lemmas  \ref{lm:x}, \ref{lm:yw}, and \ref{lm:mono}, each   lower bound \eqref{rilb}  for $r_i$ gives us $\|A_ix_i^k-A_ix_i^+\|\to 0$ as $k\rightarrow \infty$. We will first show \eqref{rilb}   for $r_p$. Then, after we do the same for $r_{p-1},\ldots,r_{i+1}$, we will get $\|A_jx_j^k-A_jx_j^+\|\to 0$ for  $j=p,p-1,\ldots,i+1$, using which we will get the lower bound \eqref{rilb} for the next $r_{i}$. We must take this backward order since $\rho_{i}^k$ (see \eqref{condpi1}) includes the terms $A_jx_j^k-A_jx_j^+$ for $j=p,p-1,\ldots,i+1$\cut{ and we must bound it in the proof}.

Our proof for each $i$ is divided into two cases.
In Case 1, $f_i$'s are restricted prox-regular (cf. Definition \ref{def:quadcvx}),  we will get \eqref{rilb} for $r_i$ by validating the condition \eqref{quadcvx1} in  Lemma \ref{lm:x} part 4 for $f_i$. In Case 2, $f_i$'s are piecewise linear (cf. Definition \ref{def:piecewise_linear}), we will show that  \eqref{quadcvx1} holds for $\gamma_i=0$ for $k\ge k_{\mathrm{pl}}$, and following the proof of Lemma \ref{lm:x} part 4, we directly get \eqref{rilb} with $\gamma_i=0$.
\medskip

\noindent \textbf{Base  step}, take $i=p$.

\noindent \textit{Case 1)} $f_p$ is restricted prox-regular. At $i=p$, the inclusion \eqref{condpi1} simplifies to
\beq\label{psubgrad}
d^k_p:=-\big(\nabla_{p}g(\vx^+)+A_p^Tw^+\big)-\beta A_p^T(By^k-By^+)\in  \partial f_p(x^+_p).
\eeq
By Lemma \ref{lm:mono} part 3 and the Lipschitz continuity of $\nabla g$, there exists a constant $M>M^*$ (independent of $\beta$) such that $$\|\nabla_{p}g(\vx^+)+A_p^Tw^+\|\le M-1.$$
By Lemma \ref{lm:By}, there exists $k_p\in\mathbb{N}$ such that, for $k>k_p$, $$\beta \|A_p^T(By^k-By^+)\|\le 1.$$
Then, we apply the triangle inequality to \eqref{psubgrad} to obtain $$\|d_p^k\|\le\|\nabla_{p}g(\vx^+)+A_p^Tw^+\|+\beta \|A_p^T(By^k-By^+)\| \le M.$$
Use this $M$ to define $S_M$ in Definition \ref{def:quadcvx}, which qualifies $f_p$ for \eqref{quadcvx} and thus validates the assumption in  Lemma \ref{lm:x} part 4, proving the lower bound \eqref{rilb} for $r_p$. As already argued, we get $\lim_{k\rightarrow\infty}\|A_px^k_p-A_px^{+}_p\|= 0$.

\smallskip
\noindent \textit{Case 2)}: $f_i$'s are piecewise linear (cf. Definition \ref{def:piecewise_linear}).  From  $\|By^k - By^+\|\rightarrow 0$ and $\|w^k - w^+\|\rightarrow 0$ (Lemma \ref{lm:By}) and $\|\nabla g(\vx^k) - \nabla g(\vx^+)\| < (p+1)\bar{M}L_g\epsilon_0$ (Lemma \ref{lm:pl}). In light of \eqref{psubgrad},
$d_p^k\in \partial f_p(x_p^+)$, $d_p^{+}\in\partial f_p(x_p^{k+2})$ such that $\|d_p^+ - d_p^k\| < 2(p+1)\bar{M}L_g \epsilon_0$ for all sufficiently large $k$.

Note that $\epsilon_0>0$ can be \emph{arbitrarily} small. Given $d_p^k\in \partial f_p(x_p^+)$ and $d_p^{+}\in\partial f_p(x_p^{k+2})$, when the following two properties both hold:  (i) $\|d_p^+ - d_p^k\| < 2(p+1)\bar{M}L_g \epsilon_0$ and (ii) $\|x^+_p - x^k_p\| < \bar{M}\epsilon_0$ (Lemma \ref{lm:pl} part 1), we can conclude that $x^+_p$ and $x^k_p$ belongs to the same $\overline{U}_{j}$. Suppose $x^+_p\in \overline{U}_{j_1}$ and  $x^k_p\in \overline{U}_{j_2}$. Because of (ii), the polyhedron $U_{j_1}$ is adjacent to the polyhedron $U_{j_2}$ or $j_1 = j_2$.
If $\overline{U}_{j_1}$ and $\overline{U}_{j_2}$ are adjacent ($j_1\not=j_2$)  and $a_{j_1}=a_{j_2}$, then we can concatenate $\overline{U}_{j_1}$ and $\overline{U}_{j_2}$ together and all the following analysis carries through.
If $\overline{U}_{j_1}$ and $\overline{U}_{j_2}$ are adjacent ($j_1\not=j_2$)  and $a_{j_1}\not=a_{j_2}$, then property (i) is only possible if at least one of $x^+_p,x_p^k$ belongs to their intersection $\overline{U}_{j_1}\cap\overline{U}_{j_2}$ so we can include both points in either $\overline{U}_{j_1}$ or $\overline{U}_{j_2}$, again giving us $j_1=j_2$. Since $x^+_p,x^k_p\in\overline{U}_{j_1}$ and $d_p^k\in \partial f_p(x_p^+)$, from the convexity of the linear function, we have
$$f_p(x^k_p)- f_p(x^+_p)-\dotp{d_p^k,x^k_p-x^+_p} \ge 0,$$
which strengthens the inequality   \eqref{quadcvx1} for $i=p$ with $\gamma_p=0$. By following the proof for Lemma \ref{lm:x} part 4, we get the lower bound \eqref{rilb} for $r_p$ with $\gamma_p=0$.
As already argued, we get $\lim_{k \rightarrow \infty}\|A_px^k_p-A_px^{+}_p\|= 0$.
\smallskip

\noindent \textbf{Inductive step}, let  $i\in\{p-1,\ldots,1\}$ and  make the inductive assumption: $\lim_{k \rightarrow \infty}\|A_jx^k_j-A_jx^{+}_j\|= 0$, $j=p,\ldots,i+1$, which together with $\lim_{k \rightarrow \infty}\|By^k - By^+\|= 0$ (Lemma \ref{lm:By}) gives $\lim_{k \rightarrow \infty}\rho_i^k=0$ (defined in \eqref{condpi1}).

\textit{Case 1)} $f_i$ is restricted prox-regular. From \eqref{condpi1}, we have
\beq\label{isubgrad}
d^k_{i}=-\big(\nabla_{i}g(x_{<i}^+,x_i^+,x_{>i}^k)+A_p^Tw^+\big)-\beta \rho_i^k\in  \partial f_i(x^+_i).
\eeq
Following a similar argument in the case $i=p$ above, there exists $k_i\in\mathbb{N}$ such that, for $k>\max\{k_p,k_{p-1},\ldots,k_i\}$, we have$$\|d_i^k\|\le \|\nabla_{i}g(x_{<i}^+,x_i^+,x_{>i}^k)+A_p^Tw^+\|+\beta \|\rho_i^k\|\le M.$$
Use this $M$ to define $S_M$ in Definition \ref{def:quadcvx} for $f_i$ and thus validates the assumption in  Lemma \ref{lm:x} part 4 for $f_i$. Therefore, we get the lower bound \eqref{rilb} for $r_i$ and thus $\lim_k\|A_ix^k_i-A_ix^{+}_i\|= 0$.

\textit{Case 2)}: $f_i$'s are piecewise linear (cf. Definition \ref{def:piecewise_linear}). The argument is the same as in the base step for case 2, except at its beginning we must use $d^k_{i}$ in \eqref{isubgrad} instead of $d_p^k$ in \eqref{psubgrad}. Therefore, we omit this part.

\cut{If $f_i$ is restricted prox-regular, by Lemma \ref{lm:mono} and the convergence of $(A_jx_j^k-A_jx_j^+)$, $j=i+1,\ldots,p$, and $(By^k-By^+)$, for sufficiently large $k_i$, we have $\|d_i\|<M$ for all $k>k_i$. Thus, the points $x^k_i$, $k>k_i$, are admissible to the inequality $\|A_ix^k_i-A_ix_i^{k+1}\|\to 0$ applied to $f_i$. Following the same arguments like above, we have $r_i \ge \frac{\beta-\gamma_i  \bar{M}^2-L_g \bar{M}^2}{2} \|A_ix^k_i-A_ix^{+}_i\|^2$ and $\|A_ix^k_i-A_ix^{+}_i\|\to0$ as $k\to\infty$. The other case where $f_i$ is piecewise linear is very similar, thus omitted.}

\emph{Finally,} by combining $r_i \ge \frac{\beta-\gamma_i  \bar{M}^2-L_g \bar{M}^2}{2} \|A_ix^k_i-A_ix^{+}_i\|^2$, for $i=1,\ldots,p$, with Lemmas \ref{lm:x} and \ref{lm:yw}, we establish the sufficient descent property \ref{P2}.

\qed
\end{proof}

\begin{lemma}[Subgradient bound property \ref{P3}]\label{lm:p3} Algorithm \ref{Alg:BCD} satisfies Property \ref{P3}.
\end{lemma}
\begin{proof}
 Because $f(\vx) = g(\vx) + \sum_{i=1}^p f_i(x_i)$, we know
  \[\partial \Lag_\beta(\vx^{+},y^{+},w^{+}) = \left(\left\{\frac{\partial \Lag_\beta}{\partial x_i}\right\}_{i=1}^p,\nabla_y \Lag_\beta,\nabla_w \Lag_\beta\right)(\vx^{+},y^{+},w^{+}).\]
  In order to prove the lemma, we only need to show that each block of $\partial \Lag_\beta$ can be controlled by some constant depending on $\beta$. Therefore, it suffices to prove
  for $s=0,\ldots, p$, there exists $d_s\in \frac{\partial {\cal L}_{\beta}}{\partial x_s}(\vx^{+},y^{+},w^{+})$ such that
\begin{equation}\label{Eq:relative_error_to_prove_x}
     \|d_s\|\leq (\sigma_{\max}(A_s)\beta + L_h\bar{M} + \sigma_{\max}(A_s)C) \left(\sum_{i=1}^p\|A_i x_i^{+}-A_i x_i^k\|+\|B y^{+}-B y^k\|\right),
\end{equation}
and
\begin{align}\label{Eq:relative_error_to_prove_w}
 \|\nabla_w \Lag_{\beta}(\vx^{+},y^{+},w^{+})\|&\leq \frac{C}{\beta} \|B y^{+}-B y^k\|,\\
\label{Eq:relative_error_to_prove_y}
     \|\nabla_y \Lag_{\beta}(\vx^{+},y^{+},w^{+})\|&\leq L_h \bar{M}\|B y^{+}-B y^k\|.
\end{align}
In order to prove \eqref{Eq:relative_error_to_prove_w}, we have
$\nabla_w \Lag_\beta (\vx^{+},y^{+},w^{+}) = \vA\vx^{+}+B y^{+}=\frac{1}{\beta}(w^{+}-w^k).$
By Lemma \ref{lemma:w_derivative}, $\|\nabla_w \Lag_\beta (\vx^{+},y^{+},w^{+})\|  \leq \frac{C}{\beta}\|B y^{+}-B y^k\|$.
In order to prove \eqref{Eq:relative_error_to_prove_y}, notice that
$\nabla_y \Lag_\beta (\vx^{+},y^{+},w^{+}) = B^T(w^{+} - w^{k})$ and apply Lemma \ref{lemma:w_derivative}.
In order to prove \eqref{Eq:relative_error_to_prove_x}, observe that for $s = 0,\ldots, p$,
  \begin{align}
  &\frac{\partial \Lag_\beta}{\partial x_s}(\vx^{+},y^{+},w^{+}) \nonumber\\
  &= \nabla_{s}g(x^{+}) +\partial f_s(x^{+}_s) + A_s^Tw^{+} + \beta A_s^T\left(Ax^{+}+B y^{+}\right)\\\label{Eq:x_err_partI}
  &=\nabla_{s}g(x^{+}_{\leq s},x^{k}_{>s}) +\partial f_s(x^{+}_s)+A_s^Tw^{k} + \beta A_s^T\left(A_{\leq s}x^{+}_{\leq s}+A_{>s}x^{k}_{>s} + By^k\right)\\
  &\quad +A_s^T(w^{+} - w^{k}) +\beta A_s^T\left(A_{>s}x^{+}_{>s}-A_{>s}x^{k}_{>s} + By^+ - By^k\right) +\nabla_{s}g(x^{+})- \nabla_{s}g(x^{+}_{\leq s},x^{k}_{>s}).\label{Eq:x_err_partII}
  \end{align}
  For the parenthesized term in \eqref{Eq:x_err_partI}, the first order optimal condition for $x^{+}_s$ yields
  $$0\in \nabla_{s}g(x^{+}_{\leq s},x^{k}_{>s}) +\partial f_s(x^{+}_s)+A_s^Tw^{k} + \beta A_s^T\left(A_{\leq s}x^{+}_{\leq s}+A_{>s}x^{k}_{>s} + By^k\right).$$
Thus for $s = 0,\ldots, p$, we can have $d_s$ as in \eqref{Eq:ds},
\begin{align}\label{Eq:ds}
d_s & := \left(A_s^T(w^{+} - w^{k}) +\beta A_s^T\left(A_{>s}x^{+}_{>s}-A_{>s}x^{k}_{>s} + By^+ - By^k\right) +\nabla_{s}g(x^{+})- \nabla_{s}g(x^{+}_{\leq s},x^{k}_{>s})\right) \\
&\in \frac{\partial \Lag_\beta}{\partial x_s}(\vx^{+},y^{+},w^{+}).\nonumber
\end{align}
Note that for any $s$, $x^k_{0}$ does not appear in any $d_s$. $w^{+} - w^{k}$, $A_{>s}x^{+}_{>s}-A_{>s}x^{k}_{>s}$, $By^+ - By^k$, and $\nabla_{s}g(x^{+})- \nabla_{s}g(x^{+}_{\leq s},x^{k}_{>s})$ can all be bounded by $\left( \sum_{i=1}^p\|A_ix_i^{+}-A_i x_i^k\|+\|By^{+} - B y^k\|\right)$. Therefore,
if we define the largest singular value of $A_s$ to be $\sigma_{\max}(A_s)$, we have the bound for $d_s$:
\[\|d_s\|\leq \left(\sigma_{\max}(A_s)\beta + L_h\bar{M} + \sigma_{\max}(A_s)C\right) \left( \sum_{i=1}^p\|A_ix_i^{+}-A_i x_i^k\|+\|By^{+} - B y^k\|\right).
\]
 That completes the proof.
\qed
\end{proof}
\begin{proof}[of Theorem \ref{thm:main}].

Lemmas \ref{lm:yw}, \ref{lm:p2}, and \ref{lm:p3} establish the properties \ref{P1}--\ref{P3}. In order to show \ref{P4}, we first note that $\Lag_\beta(\vx^{k_s}, y^{k_s}, w^{k_s})$ is monotonic nonincreasing due to Lemma \ref{lm:p2}, which implies the convergence of $\Lag_\beta(\vx^{k_s}, y^{k_s}, w^{k_s})$. Since $\Lag_{\beta}$ is lower semicontinuous (l.s.c.), we have $\lim_{s\rightarrow \infty}\Lag_\beta(\vx^{k_s}, y^{k_s}, w^{k_s})\geq \Lag_\beta(\vx^*, y^*, w^*)$.
Because the only potentially discontinuous terms in $\Lag_{\beta}$ are $f_0, \ldots, f_p$, we have
\[
\lim_{s\rightarrow \infty}\Lag_\beta(\vx^{k_s}, y^{k_s}, w^{k_s}) - \Lag_\beta(\vx^*, y^*, w^*) \leq \sum_i\limsup_{s\rightarrow \infty} f_i(x_i^{k_s}) - f_i(x_i^*).
\]
However, because $x^{k_s}_i$ is the optimal solution for the sub-problem $\min_{x_i}\Lag_{\beta}(x^{k_s}_{<i}, x_i, x^{k_s - 1}_{>i}, y^{k_s - 1}, w^{k_s - 1})$, we know $\Lag_{\beta}(x^{k_s + 1}_{<i}, x_i^*, x^{k_s}_{>i}, y^{k_s - 1}, w^{k_s - 1}) \geq \Lag_{\beta}(x^{k_s}_{\leq i}, x^{k_s - 1}_{> i}, y^{k_s - 1}, w^{k_s - 1})$. Taking the limit over their difference, we have
$ \limsup_{s\rightarrow \infty} f_i(x_i^{k_s}) - f_i(x_i^*) \leq 0$. That completes the proof for \ref{P4}.
Theorem \ref{thm:main} follows from Proposition \ref{prop:convergence}.
\qed
\end{proof}
\section{Discussion}\label{sec:discussion}
\subsection{Tightness of assumptions}
In this section, we demonstrate the tightness of the assumptions in Theorem \ref{thm:main} and compare them with related recent works. We only focus  on results that do \emph{not} make assumptions on the iterates themselves.

Hong et al. \cite{hong2014a} uses $\nabla h(y^k)$ to bound $w^k$. This   inspired our analysis. They studied ADMM for nonconvex consensus and sharing problem. Their formulation is
\begin{align*}
\Min_{x_0,\ldots, x_p, y} & \sum_{i=0}^p f_i(x_i) + h(y)\\
\mathrm{subject~to~} & \sum_{i=0}^p A_i x_i - y = 0.
\end{align*}
where $h$ is Lipschitz differentiable, $A_i$ has full column rank and $f_i$ is Lipschitz differentiable or convex for $i = 0,\ldots, p$.
Moreover, $\mathrm{dom}(f_i)$ is required to be a closed bounded set for $i = 0,\ldots, p$.

The boundedness of $\dom(f_i)$ implies the assumption \ref{A_coercive}. The requirement of $A_i$ for $i = 1,\ldots, p$ and $B = -\mathbf{I}$ implies \ref{A_feasible} and \ref{A_rank}. Moreover, $f(x_0,\ldots, x_p) = \sum_i f_i(x_i)$, which clearly implies \ref{A_f}. $h$ satisfies \ref{A_h}, too. This shows our theorem could fully cover their case.

Wang et al. \cite{Wang2015} studies the so-called Bregman ADMM and includes the standard ADMM as a special case. The following formulation is considered:

\begin{align*}
\Min_{x_0,\ldots, x_p, y} & \sum_{i=0}^p f_i(x_i) + h(y)\\
\mathrm{subject~to~} & \sum_{i=0}^p A_i x_i + B y = 0.
\end{align*}

By setting all the auxiliary functions in their algorithm to zero, their assumptions for the standard ADMM reduce to
 \begin{enumerate}
 \item[(a)] $B$ is invertible.
 \item[(b)] $h$ is Lipschitz differentiable and lower bounded.  There exists $\beta_0>0$ such that $h - \beta_0 \nabla h$ is lower bounded.
 \item[(c)] $f = \sum_{i=0}^p f_i(x_i)$ where $f_i$, $i=0,\ldots,p$ is strongly convex.
 \end{enumerate}
It is easy to see that (a), (b) and (c) imply assumptions \ref{A_coercive} and \ref{A_rank}, (a) implies \ref{A_feasible}, (c) implies \ref{A_f} and (b) implies \ref{A_h}. Therefore, their assumptions are stronger than ours. We have much more relaxed conditions on $f$, which can have a coupled Lipschitz differentiable term with separable restricted prox-regular or piecewise linear parts. We also have a simpler assumption on the boundedness without using $h-\nabla h$.

Li and Pong \cite{Li2014} studies ADMM and its proximal version for nonconvex objectives. Their formulation is
\begin{align*}
\Min_{x_0, y} & f_0(x_0) + h(y)\\
\mathrm{subject~to~} & x_0 + B y = 0.
\end{align*}
Their assumptions for ADMM are
\begin{enumerate}
\item[(1)] $f_0$ is lower semi-continuous.
\item[(2)] $h\in C^2$  with bounded Hessian matrix $c_2 I \succeq \nabla^2 h \succeq c_1 I$ where $c_2>c_1>0$.
\item[(3)] $B$ is full row rank.
\item[(4)] $h$ is coercive and $f_0$ is lower bounded.
\end{enumerate}
The assumptions (3) and (4) imply our assumptions \ref{A_coercive} and \ref{A_f}, (3) implies \ref{A_feasible} and \ref{A_rank}, and (2) implies \ref{A_h}. Our assumptions on $h$ and the matrices $A,B$ are more general.


 In summary, our convergence conditions for ADMM on nonconvex problems are the most general to the best of our knowledge. It is natural to ask whether our assumptions can be further weakened. We will provide some examples to demonstrate that, while \ref{A_coercive}, \ref{A_f} and \ref{A_rank} can probably be further weakened, \ref{A_h} and \ref{A_feasible} are  essential in the convergence of nonconvex ADMM and cannot be completely dropped in general.
In \cite{Chen2014}, their divergence example is
\begin{subequations}
 \begin{align}
 \Min_{x_1,x_2,y}~ & \hspace{50pt}0\\
 \text{subject to } &
 \left(
 \begin{array}{c}
 1\\
 1\\
 1
 \end{array}
 \right)x_1 +
  \left(
 \begin{array}{c}
 1\\
 1\\
 2
 \end{array}
 \right)x_2 +
  \left(
 \begin{array}{c}
 1\\
 2\\
 2
 \end{array}
 \right)y =
  \left(
 \begin{array}{c}
 0\\
 0\\
 0
 \end{array}
 \right).
 \end{align}
\end{subequations}
Another related example is shown in \cite[Example 7]{Li2014}.
\begin{subequations}
 \begin{align}
 \Min_{x_1,x_2,y}~ & \iota_{S_1}(x_1) + \iota_{S_2}(x_2)\\
 \text{subject to } & x_1 = y\\
 & x_2 = y,
 \end{align}
\end{subequations}
 where $S_1 = \{x = (x_1,x_2)\mid x_2 = 0\}$, $S_2 = \{(0,0),(2,1),(2,-1)\}$.
 These two examples satisfy \ref{A_coercive} and \ref{A_f}-\ref{A_h} but fail to satisfy \ref{A_feasible}. Without \ref{A_feasible}, for come cases ADMM is incapable to find a feasible point at all, let alone a stationary point. Therefore, \ref{A_feasible} is  indispensable.

To see the necessity of \ref{A_h} (the smoothness of $h$), consider another divergence example
\begin{subequations}
 \begin{align}
 \Min_{x,y}~ & -|x| + |y|\\
 \text{subject to } & x = y,\ x\in [-1,1].
 \end{align}
\end{subequations}
For any $\beta>0$, with the initial point $(x^0,y^0,w^0) = (-\frac{2}{\beta},0,-1)$, we get the sequence $(x^{2k+1},y^{2k+1},w^{2k+1}) = (\frac{2}{\beta},0,1)$ and $(x^{2k},y^{2k},w^{2k}) = (-\frac{2}{\beta},0,-1)$ for $k\in \mathbb{N}$, which diverges. This problem satisfies all the assumptions except \ref{A_h}, without which $w^k$ cannot be controlled by $y^k$ anymore. 
Therefore, \ref{A_h} is also indispensable.

Although the assumptions \ref{A_feasible} and \ref{A_h} seem essential for the convergence of ADMM, other assumptions, especially the assumption \ref{A_f}, might be further relaxed. Moreover, our result requires the $y$-block to be updated at last right before the update of multiplier. Further studies could be carried out to study the case when a different order is used.
\subsection{Primal variables' update order in ADMM}
We discuss about the update order of $\{x_i\}_{i=0}^p$ and $y$ in this subsection.  Theorem \ref{thm:main} and Theorem \ref{thm:main_theorem_convex} apply to the ADMM\ in which the primal variables $x_0,\ldots, x_p$ are sequentially updated in a fixed order. 
With minor changes to the proof, both theorems still hold for arbitrary update orders of $x_1,\ldots, x_p$, possibly different between iterations, as long as $x_0$ is always the first and $y$ is always the last primal variable to update, just before $w$. In particular,
$x_1,\ldots, x_p$ could be updated using random scheme or greedy scheme, which may help avoid low-quality local solutions. 
Recently, randomized ADMM is considered in papers such as \cite{Iutzeler2013,ARock2016,Sun2015}. \cite{Sun2015} considered the randomly permuted ADMM for solving linear systems, and proved its convergence in the expectation sense. However, in general, permutation including the last block $y$ could cause ADMM to diverge (A convex example can be found in \cite{Yin2016}). Consider
\begin{align*}
\Min_{x,y\in\RR} &\quad x(1+y) \\
\St &\quad x - y=0.
\end{align*}
It is easy to check that, if we fix the update order to either $x, y, w$  or $y, x, w$ for all iterations, Algorithm \ref{Alg:BCD} converges. However, if we alternate between the two update orders, we obtain (with $\alpha := 1/\beta$) the diverging sequence
$(x^{2k+1}, y^{2k+1}, w^{2k+1}) = (2\alpha(\alpha - 1), -\alpha,\alpha -1)$ and $(x^{2k}, y^{2k}, w^{2k}) = (-\alpha, 2\alpha(\alpha - 1),-\alpha)$.  Another divergent example when primal variables' update order alternates is the following convex and nonsmooth problem:
\begin{subequations}
 \begin{align}
 \Min_{x,y}~ & 2|x - 1| + |y|\\
 \text{subject to } & x = y.
 \end{align}
\end{subequations}
\subsection{Inexact optimization of subproblems}
Note that all subproblems in Algorithm \ref{Alg:BCD} should be solved exactly. This might restrict the wide use of the algorithm in real applications. Thus, the convergence of the inexact version of Algorithm \ref{Alg:BCD} is discussed here. We extend the developed convergence results to the following inexact version of Algorithm \ref{Alg:BCD} under some additional assumptions. More specifically, we assume that the sequence $\{\vx^k, y^k, w^k\}$ generated by the inexact version of Algorithm \ref{Alg:BCD} satisfies
\begin{enumerate}[label={P\arabic*'}]
\item\label{P1'} \textbf{(Boundedness)} $\{\vx^k, y^k, w^k\}$ is bounded, and $\Lag_\beta(\vx^k,y^k,w^k)$ is lower bounded;
\item\label{P2'} \textbf{(Sufficient descent)} there is a nonnegative sequence $\{\eta_k\}$ and a constant $C_1>0$ such that for all sufficiently large $k$, we have
\begin{align}\label{Cond:1'}
&\Lag_\beta(\vx^{k},y^k,w^{k}) - \Lag_\beta(\vx^{k+1},y^{k+1},w^{k+1}) \geq C_1\big(\|B (y^{k+1} - y^k)\|^2 +\sum_{i=1}^p \|A_i(x_i^{k}-x_i^{k+1})\|^2\big) - \eta_k,
\end{align}
\item\label{P3'} \textbf{(subgradient bound)} and there exists $d^{k+1}\in \partial \Lag_\beta(\vx^{k+1},y^{k+1},w^{k+1})$ such that
\begin{align}\label{Cond:2'}
\|d^{k+1}\| &\leq C_2\big(\|B(y^{k+1}-y^k)\|+\sum_{i=1}^p\|A_i(x_i^{k+1}-x_i^k)\|\big) + \eta_k.
\end{align}
\cut{where $\sum_k \eta_k < \infty$.}
\end{enumerate}
When $\sum_k \eta_k < \infty$, the convergence results in Theorem \ref{thm:main} still hold for this sequence. This is because Proposition \ref{prop:convergence} still holds when the error is summable.
However, when a specific algorithm is applied to solve these subproblems inexactly, it might require some additional conditions, and we leave this in the future work.

\section{Applications}\label{sec:application}
In this section, we apply the developed  convergence results to several well-known applications. To the best of our knowledge, all the obtained convergence results are novel and cannot be recovered from the previous literature.

\subsection*{A) Statistical learning}
\cut{A typical statistical learning model assumes the underlying distribution of the data and tries to estimate the distribution based on some samples. The distribution is usually characterized by a set of parameters. To find those parameters, an optimization is usually proposed.} Statistical learning models often involve two terms in the objective function. The first term is used to measure the fitting error. The  second term is a regularizer to control the model complexity\cut{ and lessen overfitting}. Generally speaking, it can be written as
\begin{align}\label{Eq:original_stat_learning}
  \Min_{ x\in\RR^n}  &\quad\sum_{i=1}^p l_i(A_i x-b_i) + r( x),
\end{align}
where $A_i\in \RR^{m_i\times n}$, $b_i\in \RR^{m_i}$ and $ x\in \RR^n$. Examples of the fitting measure
$l_i$ include least squares, logistic functions, and other smooth functions. The regularizers can be some \cut{sparse} sparsity-inducing functions \cite{Bach2012,Daubechies2010,Yang2015,Zeng2014,Zeng2015,ZengAIT2015,ZengHalf2014} such as MCP, SCAD, $\ell_q$ quasi-norms for $0 < q \leq 1$.
Take LASSO as an example,
\[\Min_ x~ \frac{1}{2}\|y - A x\|^2 + \lambda \| x\|_1.\]
The first term $\|y - A x\|^2$ measures the difference between the linear model $A x$ and outcome $y$. The second term $\| x\|_1$ measures the sparsity of $ x$.

In order to solve \eqref{Eq:original_stat_learning} using ADMM, we reformulate it as
\begin{align}\label{Eq:Statistic}
  \Min_{ x,\{z_i\}_{i=1}^p}  &\quad r( x) + \sum_{i=1}^p l_i(A_i z_i-b_i) ,\\
  \text{subject to}  &\quad x = z_i, \ \forall i=1,\ldots,p.\nonumber
\end{align}
\cut{In \cite{Breheny2011}, ADMM is applied to solve this model because it can be implemented parallelly \cite{Bertsekas1989,Goldstein2015}, which is very suitable for large data sets. }
Algorithm \ref{Alg:Statistic} gives the standard ADMM algorithm for this problem.

\begin{algorithm}
{\small
Denote $\mathbf{z} = [z_1;z_2;\cdots;z_p]$, $\mathbf{w} = [w_1;w_2;\cdots;w_p]$.
\begin{algorithmic}\caption{ADMM for \eqref{Eq:Statistic}}\label{Alg:Statistic}
\STATE {\bf Initialize} $ x^0, \mathbf{z}^0, \mathbf{w}^0$ arbitrarily;
\WHILE{stopping criterion not satisfied}
\STATE
$ x^{k+1} \gets \argmin_{x} r(x) + \frac{\beta}{2}\sum_{i=1}^p (z_i^k+ \frac{w_i^k}{\beta}-x)^2$;
\FOR{$s=1,\ldots,p$}
\STATE  $z^{k+1}_{s} \gets \argmin_{z_s} l_s(A_s z_s-b_s) + \frac{\beta}{2} (z_s+\frac{w^k_s}{\beta}- x^{k+1})^2$;
\STATE $w^{k+1}_{s} = w^k_s + \beta(z^{k+1}_s -  x^{k+1})$;
\ENDFOR
\STATE $k\gets k+1$;
\ENDWHILE
\STATE return $ x^k$.
\end{algorithmic}}
\end{algorithm}
Based on Theorem \ref{thm:main}, we have the following corollary.
\begin{coro}
  Let $r(x)=\| x\|_q^q = \sum_i |x_i|^q$, $0 < q \leq 1$, SCAD, MCP, or any piecewise linear function, if
  \begin{enumerate}
    \item[i) ] (Coercivity) $r(x) + \sum_i l_i(A_i x + b_i)$ is coercive;
    \item[ii) ] (Smoothness) For each $i=1,\ldots, p$, $l_i$ is Lipschitz differentiable.
  \end{enumerate}
then for sufficiently large $\beta$, the sequence $(x^k,\mathbf{z}^k,\mathbf{w}^k)$generated by Algorithm \ref{Alg:Statistic} has limit points and all of its limit points are stationary points of the augmented Lagrangian $\Lag_{\beta}$.
\end{coro}
\begin{proof}
Rewrite the optimization to a standard form, we have
\begin{subequations}\label{lassoreform}
\begin{align}
  \Min_{ x,\{z_i\}_{i=1}^p}  &\quad r(x)+\sum_{i=1}^p l_i(A_i z_i-b_i),\\
  \text{subject to}  &\quad
  E x + \mathbf{I}_{np}{{\bf z}} = 0.
\end{align}
\end{subequations}
where
$E = -[\mathbf{I}_n;\ldots;\mathbf{I}_n]\in \RR^{np\times n}$, $\mathbf{I}_{np}\in \RR^{np\times np}$ is the identity matrix, and ${{\bf z}} = [z_{1};\ldots;z_{p}]\in \RR^{np}$. Fitting \eqref{lassoreform} to the standard form \eqref{Eq:originalForm1}, there are two blocks $(x,{{\bf z}})$ and $B = \mathbf{I}_{np}$. $f(x) = r(x)$ and $h(z) = \sum_{i=1}^p l_i(A_i z_i- b_i)$.

Now let us check \ref{A_coercive}--\ref{A_h}.
\ref{A_coercive} holds because of i). \ref{A_feasible} holds because $B = \mathbf{I}_{np}$. \ref{A_h} holds because of ii). \ref{A_rank} holds because $E$ and $\mathbf{I}_{np}$ both have full column ranks.
{ If $r(x)$ is piecewise linear, then \ref{A_f} holds naturally.} If $r(x)$ is MCP
\[
P_{\gamma,\lambda}(x) \triangleq
\left\{
\begin{array}{ll}
\lambda|x| - \frac{x^2}{2\lambda}, &\quad \text{if~}|x|\leq \gamma\lambda\\
\frac{1}{2}\gamma\lambda^2, &\quad \text{if~}|x|\geq \gamma\lambda
\end{array}
\right.,
\]
 or SCAD
\[
Q_{\gamma,\lambda}(x) \triangleq
\left\{
\begin{array}{ll}
\lambda |x|, & \quad \text{if~} |x| \leq \lambda\\
\lambda|x| - \frac{2\gamma\lambda|x| - x^2 - \lambda^2}{2\gamma - 2}, & \quad \text{if~}\lambda < |x|\leq \gamma\lambda\\
\frac{1}{2}(\gamma+1)\lambda^2, & \quad \text{if~}|x|\geq \gamma\lambda
\end{array}
\right..
\]
we can verify that those functions are the maximum of a set of quadratic functions.
 Then by \cite[Example 2.9]{poliquin1996prox}, we know they are prox-regular.
Hence, it remains to verify \ref{A_f}{(ii)a} that $r(x) = \sum_{i} |x_i|^q$ is restricted prox-regular. When $q = 1$, this is trivial; when $0 < q < 1$, it has been proved in Proposition \ref{Prop:restricted_prox}.
This verifies \ref{A_f}  and completes the proof. \qed
\end{proof}
\subsection*{B) Minimization on compact manifolds}
Compact manifolds and their projection operators such as spherical manifolds $S^{n-1}$, Stiefel manifolds (the set of $p$ orthonormal vectors $x_1,\ldots,x_p\in\RR^n$, $p\le n$) and Grassmann manifolds (the set of subspaces in $\RR^n$ of dimension $p$) often arise in optimization.
Some recent studies and algorithms can be found in \cite{Wen2013,Lai2014,Lu2012Augmented}. A simple example is:
\begin{align}\label{Eq:Orthogonal}
  \Min_{x\in\RR^n}  &\quad J(x),\\
  \text{subject to}  &\quad\|x\|^2 = 1,\nonumber
\end{align}
More generally, let $S$ be any compact set. We consider the problem
\begin{align}\label{Eq:OrthogonalS}
  \Min_{x\in\RR^n}  &\quad J(x),\\
  \text{subject to}  &\quad x\in S,\nonumber
\end{align}
which can be rewritten to the following form:
\begin{subequations}\label{Eq:Orthogonal_reform}
\begin{align}
  \Min_{x,y}  &\quad \iota_{S}(x) + J(y),\\
  \text{subject to}  &\quad x - y = 0,
\end{align}
\end{subequations}
where $\iota_S(\cdot)$ is the indicator function: $\iota_S(x) = 0$ if $x\in S$ or $\infty$ if $x\not\in S$.
Applying ADMM to solve this problem, we get Algorithm \ref{Alg:Orthogonal}.
\begin{algorithm}
{\small
\begin{algorithmic}\caption{ADMM for minimization on a compact set \eqref{Eq:Orthogonal_reform}}\label{Alg:Orthogonal}
\STATE {\bf Initialize} $x^0,y^0,w^0$ arbitrarily;
\WHILE{stopping criterion not satisfied}
\STATE  $x^{k+1} \gets \Proj_S(y^k -  \frac{w^k}{\beta})$; 
\STATE
$y^{k+1} \gets \argmin_{y} J(y) + \frac{\beta}{2} \|y -  \frac{w^k}{\beta}-x^{k+1}\|^2$;
\STATE $w^{k+1} \gets w^k + \beta(y^{k+1} - x^{k+1})$;
\STATE $k\gets k+1$.
\ENDWHILE
\STATE return $x^k$.
\end{algorithmic}}
\end{algorithm}

Based on Theorem \ref{thm:main}, we have the following corollary.
\begin{coro}
  If $J$ is Lipschitz differentiable, then for any sufficiently large $\beta$, the sequence $(x^k,y^k,w^k)$ \cut{in the} generated by Algorithm \ref{Alg:Orthogonal} has 
  at least one limit point, and each limit point is a stationary point of the augmented Lagrangian $\Lag_{\beta}$.
\end{coro}
\begin{proof}
To show this corollary, we shall verify assumptions \ref{A_coercive}--\ref{A_h}.

The assumption \ref{A_coercive} holds because the feasible set is a bounded set and $J$ is lower bounded on the feasible set.
\ref{A_feasible} and \ref{A_rank} hold because both $A$ and $B$ are identity matrices. \ref{A_h} holds because $J$ is Lipschitz differentiable. \ref{A_f} holds because $\iota_S$ is lower semi-continuous.
\end{proof}

\subsection*{C) Smooth Optimization over complementarity constraints}

We consider the following optimization problem over complementarity constraints.

\begin{align}
\label{Eq:OPToverComplementarity}
&\Min_{\{x,y\}} \quad h(x,y)\\
& \mathrm{subject\ to} \quad x^T y =0, x\geq 0, y \geq 0, \nonumber
\end{align}
where $h(x,y)$ is a smooth function with Lipschitz differentiable gradient.
The considered problem is a special case of the mathematical programming with equilibrium constraints (MPEC) \cite{Chen2016working}, and includes the linear complementarity problem (LCP) \cite{cottle-dantzig-LCP1968} as a special case. In order to apply the ADMM algorithm to solve this problem, we introduce two auxiliary variables $x', y' \in \mathbb{R}^n$ and define the complementarity set $S \triangleq \{(x,y): x^Ty=0,\ x\geq 0, \ y \geq 0\}$. With these notations, problem \eqref{Eq:OPToverComplementarity} can be reformulated as follows
\begin{align}
\label{Eq:OPToverComplementarity1}
&\Min_{\{x',y',x,y\}} \quad \iota_S(x',y')+h(x,y)\\
& \mathrm{subject\ to} \quad x'-x=0, \quad y'-y=0, \nonumber
\end{align}
where $\iota_S(x',y')$ denotes the indicator function of the set $S$.
Furthermore, let ${\bf x}_0 = \left(
 \begin{array}{c}
 x'\\
 y'
 \end{array}
 \right)$
 and the second block
${\bf y}= \left(
 \begin{array}{c}
 x\\
 y
 \end{array}
 \right)$.
 Then \eqref{Eq:OPToverComplementarity1} becomes
\begin{align}
 \label{Eq:OPToverComplementarity2}
&\Min_{{\bf x}_0,{\bf y}} \quad \iota_S({\bf x}_0)+h({\bf y})\\
& \mathrm{subject\ to} \quad {\bf x}_0 - {\bf y}=0. \nonumber
\end{align}

\begin{coro}
Assume that $h$ is Lipschitz differentiable and coercive over the complementarity set, then for sufficiently large $\beta$, the sequence $({\bf x}_0^k,{\bf y}^k,w^k)$ generated by Algorithm ADMM applied to \eqref{Eq:OPToverComplementarity2} has limit points and all of its limit points are stationary points of the augmented Lagrangian $\Lag_{\beta}$.
\end{coro}
\begin{proof}
In order to prove this corollary, we only need to verify assumptions {\bf A}1-{\bf A}5. {\bf A}1 holds for the coercivity of $h$ over $S$ and the specific form of $\iota_S$. {\bf A}2 is obvious due to in this case $A={\bf I}$ and $B=-{\bf I}$. {\bf A}3 holds for both ${\bf I}$ and $-{\bf I}$ being full column rank. {\bf A}4 can be satisfied by setting $f_0 = \iota_S$ and $g = 0$. {\bf A}5 holds due to the Lipschitz differentiability of $h$. Thus, according to Theorem 1, we complete the proof.
$\Box$
\end{proof}

\subsection*{D) Matrix decomposition}
ADMM has also been applied to solve matrix related problems, such as sparse principle component analysis (PCA) \cite{Hu2015}, matrix decomposition \cite{Slavakis2014,luo13b}, matrix completion \cite{Cai2010}, matrix recovery \cite{Oymak2011}, non-negative matrix factorization \cite{xu11,Sun2014} and background/foreground extraction \cite{Chartrand2012,Yang2015}.

In the following, we take  the video surveillance image-flow problem as an example. A video can be formulated as a matrix $V$ where each column is a vectorized image of a video frame. It can be generally decomposed into three parts, background, foreground, and noise. The background has low rank since it does not move. The derivative of the foreground is small because foreground (such as human beings, other moving objectives)  moves relatively slowly.
The noise is generally assumed to be Gaussian and thus can be modeled via Frobenius norm.

More specifically, consider the following  matrix decomposition model:
\begin{align}\label{Eq:matrixDecomposition}
  \Min_{X,Y,Z}  &\quad p(X) +  \sum_{i=1}^{m-1}\|Y_i - Y_{i+1}\| + \norm{Z}_F^2,\\
  \text{subject to}  &\quad V = X + Y + Z,
\end{align}
where $X,Y,Z,V\in \RR^{n\times m}$, $Y_i$ is the $i$th column of $Y$, $\norm{\cdot}_F$ is the Frobenius norm, and $p(X)$ is any lower bounded lower semi-continuous penalty function, for example, the Schatten-$q$ quasi-norm $\norm{X}_q$ ($0<q\leq 1$):
\[\norm{A}_{q} = \sum_{i=1}^n \sigma_i^{q}(A),\]
where $\sigma_i(A)$ is the $i$th largest singular value of $A$.


The corresponding ADMM algorithm is given in Algorithm \ref{Alg:Matrix}.

\begin{algorithm}
{\small
\begin{algorithmic}\caption{ADMM for \eqref{Eq:matrixDecomposition}}\label{Alg:Matrix}
\STATE {\bf Initialize} $Y^0,Z^0,W^0$ arbitrarily;
\WHILE{stopping criteria not satisfied}
\STATE
$X^{k+1} \gets \argmin_{X} p(X) + \frac{\beta}{2}\norm{X + Y^k + Z^k - V + W^k/\beta}_F^2;$
\STATE  $Y^{k+1} \gets \argmin_Y \sum_{i=1}^m \|Y_i - Y_j\| + \frac{\beta}{2} \norm{X^{k+1} + Y + Z^k - V + W^k/\beta}_F^2;$
\STATE  $Z^{k+1} \gets \argmin_Z \norm{Z}_F^2 + \frac{\beta}{2} \norm{X^{k+1} + Y^{k+1} + Z - V + W^k/\beta}_F^2;$
\STATE  $W^{k+1} \gets W^k + \beta( X^{k+1} + Y^{k+1} + Z^{k+1} - V);$
\STATE $k\gets k+1$;
\ENDWHILE
\STATE return $X^k,Y^k,Z^k$.
\end{algorithmic}}
\end{algorithm}

\begin{coro}
For a sufficiently large $\beta$, the sequence $(X^k,Y^k,Z^k,W^k)$ generated by Algorithm \ref{Alg:Matrix} has at least one limit point, and each limit point is a stationary point of the augmented Lagrangian function $\Lag_{\beta}$.
\end{coro}
\begin{proof}
Let us verify assumptions \ref{A_coercive}--\ref{A_h}.
The assumption \ref{A_coercive} holds because of the coercivity of $\norm{\cdot}_F$ and $\norm{\cdot}_q$. \ref{A_feasible} and \ref{A_rank} hold because all the coefficient matrices are identity matrices. \ref{A_h} holds because $\norm{\cdot}_F^2$ is Lipschitz differentiable. \ref{A_f} holds because $p$ is lower semi-continuous.

\end{proof}

\section{Conclusion}\label{sec:conclusion}
This paper studied the convergence of ADMM, in its multi-block and original cyclic update form, for nonconvex and nonsmooth optimization. The objective can be certain nonconvex and nonsmooth functions while the constraints are coupled linear equalities.
Our results theoretically demonstrate that ADMM, as a variant of ALM, may converge under weaker conditions than ALM. While ALM generally requires the objective function to be smooth, ADMM only requires it to have a smooth part $h(y)$ while the remaining part $f(\vx)$ can be coupled, nonconvex, and include separable nonsmooth functions and indicator functions of constraint sets.

Our results relax the previous assumptions (e.g., semi-convexity) and allow the nonconvex functions such as $\ell_q$ quasi-norm $(0<q<1)$, Schatten-$q$ quasi-norm, SCAD, and others that often appear in sparse optimization. They also allow nonconvex constraint sets such as unit spheres, matrix manifolds, and complementarity constraints.

 The underlying proof technique identifies an exclusion set where the sequence does not enter after finitely many iterations. We also manage to have a very general first block $x_0$. We show that while the middle $p$ blocks $x_1,\ldots, x_p$ can be updated in an arbitrary order for different iterations, the first block $x_0$ should be updated at first and the last block $y$  at last; otherwise, the concerned iterates may diverge according to the existing example.


Our results can be applied to problems in matrix decomposition, sparse recovery, machine learning, and optimization on compact smooth manifolds and lead to novel convergence guarantees.


\section*{Acknowledgements}
We would like to thank Drs. Wei Shi, Ting Kei Pong, and Qing Ling for their insightful comments, and Drs. Xin Liu and Yangyang Xu for helpful discussions. We thank the three anonymous reviewers for their review and helpful comments.

\section*{Appendix}
\begin{proof}\textbf{Proof of Proposition \ref{Prop:restricted_prox}}

The fact that convex functions and the $C^1$ regular functions are prox-regular has been proved in the previous literature, for example, see \cite{poliquin1996prox}. Here we only prove the second part of the proposition.

(1): For functions $r( x) = \sum_{i} |x_i|^q$ where $0 < q < 1$, the set of general subgradient of $r(\cdot)$ is
$$\partial r(x) = \left\{d=[d_1;\ldots;d_n]\left|d_i = q\cdot \mathrm{sign}(x_i)|x_i|^{q-1} \text{ if } x_i \neq 0 \text{; } d_i\in \RR \text{ if } x_i = 0\right.\right\}.$$
For any two positive constants $C>0$ and $M>1$, take $\gamma = \max\left\{\frac{4({n}C^q+MC)}{c^2},q(1-q)c^{q-2}\right\}$, where
{ $c \triangleq \frac{1}{3}(\frac{q}{M})^{\frac{1}{1-q}}$.} The exclusion set $S_{M}$ contains the set $\{ x|\min_{x_i\neq 0} |x_i|\leq 3c\}$. For any point $z\in \mathbb{B}(0,C)/S_{M}$ and $y\in \mathbb{B}(0,C)$, if $\|z-y\|\leq c$, then $\mathrm{supp}(z) \subset \mathrm{supp}(y)$ and $\|z\|_0 \leq \|y\|_0$, where $\mathbb{B}(0,C) \triangleq \{x| \|x\|<C\}$, $\mathrm{supp}(z)$ denotes the index set of all non-zero elements of $z$ and $\|z\|_0$ denotes the cardinality of $\mathrm{supp}(z)$. Define
 \[y'_{i} =
 \left\{\begin{array}{cc}
 y_{i} & \quad i\in \mathrm{supp}(z)\\
 0 & \quad i\not\in \mathrm{supp}(z)
 \end{array}\right.,~i=1,\ldots, p. \]
Then { for any $d\in \partial r(z)$,} the following line of proof holds,
\begin{align}\nonumber
\|y\|_q^q-\|z\|_q^q - \big<d,y-z\big>
\stackrel{(a)}{\geq} & \|y'\|_q^q-\|z\|_q^q - \big<d,y'-z\big>\\ \nonumber
 \stackrel{(b)}{\geq}& - \frac{q(1-q)}{2}c^{q-2}\|z-y'\|^2\\ \label{Eq:stat_1}
 \stackrel{(c)}{\geq}& - \frac{q(1-q)}{2}c^{q-2}\|z-y\|^2,
\end{align}
{where (a) holds for $\|y\|_q^q = \|y'\|_q^q + \|y-y'\|_q^q$ by the definition of $y'$,}
(b) holds because $r(x)$ is twice differentiable along the line segment connecting $z$ and $y'$, and the second order derivative is no bigger than $q(1-q)c^{q-2}$, and (c) holds because $\|z-y\|\geq \|z-y'\|$.
{ While if $\|z-y\|> c$, then for any $d\in \partial r(z)$, we have
\begin{equation} \label{Eq:stat_2}
 \|y\|_q^q-\|z\|_q^q-\big<d,y-z\big>\geq -(2nC^q + 2MC) \geq  -\frac{2nC^q+2MC}{c^2}\|y-z\|^2.
\end{equation}
}
Combining \eqref{Eq:stat_1} and \eqref{Eq:stat_2} yields the result.

(2):
We are going to verify that Schatten-$q$ quasi-norm $\norm{\cdot}_q$ is restricted prox-regular. Without loss of generality, suppose $A\in \RR^{n\times n}$ is a square matrix.

Suppose the singular value decomposition (SVD)
of $A$ is
\begin{equation}
A = U\Sigma V^T =
[U_1,U_2]\cdot
\left[\begin{array}{cc}
\Sigma_1 & 0\\
0      & 0
\end{array}\right] \cdot
\left[\begin{array}{c}
V_1^T\\
V_2^T
\end{array}\right],\label{Eq:SVD}
\end{equation}
where $U,V\in \RR^{n\times n}$ are orthogonal matrices, and $\Sigma_1\in \RR^{K\times K}$ is diagonal whose diagonal elements are $\sigma_i(A)$, $i=1,\ldots,K$. Then the general subgradient of $\norm{A}_q^q$ \cite{Watson1992}
 is
\[\partial \norm{A}_q^q = U_1DV_1^T + \{U_2\Gamma V_2^T\big|\text{$\Gamma$ is an arbitrary matrix}\},\]
where $D\in \RR^{K\times K}$ is a diagonal matrix whose $i$th diagonal element is $ d_i = q\sigma_i(A)^{q-1}$.

Now we are going to prove $\norm{\cdot}_q^q$ is restricted prox-regular, i.e., for any positive parameters $M, P>0$, there exists $\gamma>0$ such that for any $\norm{B}_F<P$,  $\norm{A}_F<P$, $A\not\in S_M = \{A| \forall~X\in \partial \norm{A}_q^q,~\norm{X}_F>M\}$, and  $T = U_{1} D V_{1}^T + U_{2}\Gamma V_{2}^T\in \partial \norm{A}_q^q, \norm{T}_F\leq M$, we hope to show
\begin{align}\label{Eq:matrix_target}
\norm{B}_q^q - \norm{A}_q^q - \big<T,B - A\big> \geq -\frac{\gamma}{2}\norm{A-B}^2_F.
\end{align}

Let $\epsilon_0 = \frac{1}{3}(M/q)^{1/(q - 1)}$.
If $\norm{B - A} > \epsilon_0$, we have
\begin{align}\label{Eq:matrix_part1}
\norm{B}_q^q - \norm{A}_q^q - \big<T,B - A\big> \geq& - n^2P^q - M\cdot \|B - A\|_F \geq -(M\epsilon_0^{-1}+n^2P^q\epsilon_0^{-2})\norm{A-B}_F^2.
\end{align}

If $\norm{B-A}_F<\epsilon_0$,
consider the decomposition of $B = U_B \Sigma^B V_B^T = B_1 + B_2$ where $B_1 = U_B \Sigma^B_1 V_B^T$, $\Sigma^B_1$ is the diagonal matrix preserving elements of $\Sigma^B$ bigger than $\frac{1}{3}(M/q)^{1/(q - 1)}$, and $B_2 = U_B \Sigma^B_2 V_B^T$ where $\Sigma^B_2 = \Sigma^B - \Sigma^B_1$.

Define a set $S' \triangleq\{T\in \mathbb{R}^{n \times n}|\norm{T}_F \leq P,~ \min_{\sigma_i>0} \sigma_i(T) \geq \epsilon_0\}$. Let's prove $A, B_1\in S'$. If $\min_{\sigma_i >0} \sigma_i(A) < (M/q)^{1/(q - 1)}$, then for any $X\in \partial \norm{A}_q^q$, $X = U_1DV_1^T + U_2\Gamma V_2^T$ and
\[\norm{X}_F \geq \norm{U_1DV_1^T}_F \geq  \min_{\sigma_i>0} q\sigma_i^{q-1} \geq M,\]
which contradicts with the face that $A\not\in S_M$. As for $B_1$, because of $\norm{A - B}_F\leq \epsilon_0$ and $\min_{\sigma_i >0} \sigma_i(A) < (M/q)^{1/(q - 1)}$, using Weyl inequalities will we get $B_1\in S'$.

Define the function $F:S'\subset \RR^{n\times n}\rightarrow \RR^{n\times n}$, for $A = U_1\Sigma V_1^T$, $F(A) = U_1 D V_1^T$, where
\[D = \mathrm{diag}(q\sigma_1(A)^{q-1},\ldots,q\sigma_1(A)^{q-1})\]
and ($0^{q-1} = 0$).
 Based on \cite[Theorem 4.1]{Ding2014} and the compactness of $S'$, $F(A)$ is Lipschitz continuous in $S'$, i.e., there exists $L>0$, for any two matrices $A, B\in S'$ , $\norm{F(A) - F(B)}_F\leq L\norm{A - B}_F$. This implies
\begin{align}\label{Eq:final_1}
\norm{B_1}_q^q - \norm{A}_q^q - \big<U_1DV_1^T,B_1 - A\big>
\geq & -\frac{L}{2}\norm{B_1 - A}_F^2.
\end{align}
In addition, because $\norm{U_{2}^TU_B}_F< \norm{B_1 - A}_F/\epsilon_0$ and $\norm{V_{2}^TV_B}_F < \norm{B_1 - A}_F/\epsilon_0$ (see \cite{Li2000}),
\begin{align}\label{Eq:final_2}
\big<U_{2}\Gamma V_{2}^T,B_1 - A\big>= &\big<\Gamma,U_{2}^TU_B\Sigma_BV_B^TV_{2}\big>
\geq - \frac{M^2}{\epsilon_0^2} \norm{B_1 - A}_F^2.
\end{align}
Furthermore, $\norm{B_2}_q^q - \big<T,B_2\big> \geq 0$ and $\norm{B_1 - A}_F \leq  \norm{B - A}_F+\norm{B - B_1}_F\leq 2\norm{B - A}_F$, together with \eqref{Eq:final_1} and \eqref{Eq:final_2} we have
\begin{align}\nonumber
\norm{B}_q^q - \norm{A}_q^q - \big< T,B - A\big>
= & \norm{B_1}_q^q  - \norm{A}_q^q - \big<T,B_1 - A\big> + \norm{B_2}_q^q - \big<T,B_2\big>\\
\geq &
 -\left(\frac{L}{2} + \frac{M^2}{\epsilon_0^2}\right)\norm{B_1 - A}_F^2 \geq -\left(2L + \frac{4M^2}{\epsilon_0^2}\right) \norm{B - A}_F^2.\label{Eq:matrix_part2}
\end{align}
Combining \eqref{Eq:matrix_part1} and \eqref{Eq:matrix_part2}, we finally prove \eqref{Eq:matrix_target} with appropriate $\gamma$.

(3): We need to show that the indicator function $\iota_S$ of a $p$-dimensional compact $C^2$ manifold $S$ is restricted prox-regular.
First of all, by definition, the exclusion set $S_M$ of $\iota_S$ is empty for any $M>0$.
Since $S$ is compact and $C^2$,  there are a series of $C^2$ homeomorphisms $h_\eta : \RR^{p} \mapsto \RR^n$, $\eta\in \{1,\ldots, m\}$ and $\delta>0$ such that for any $x$, there exist an $\eta$ and an $\alpha_x$ satisfying $x = h_\eta(\alpha_x)\in S$.
Furthermore, for any $\|y - x\|\leq \delta$, we can find an $\alpha_y$ satisfying $y = h_\eta(\alpha_y)$.

Note that $\partial \iota_{S}(x)= \mathrm{Im}(J_{h_\eta}(x))^\perp$, where $J_{h_\eta}$ is the Jacobian  of $h_\eta$.
For any $d\in \partial \iota_S(x)$, $\|d\| \leq M$ and $\|x-y\|\leq \delta$,
\begin{align}
\iota_S(y) - \iota_S(x) - \big<d,y - x\big> \nonumber
=& -\big<d,h_\eta(\alpha_y) - h_\eta(\alpha_x)\big> \\ \nonumber
=& -\big<d,h_\eta(\alpha_y) - h_\eta(\alpha_x) - J_{h_\eta}(\alpha_y - \alpha_x)\big> \\ \nonumber
\geq& -\|d\| \cdot \gamma\|\alpha_y - \alpha_x\|^2 \\
\geq & -M \gamma C^2 \|x - y\|^2, \label{Eq:intersection_base_1}
\end{align}
where $\gamma$ and $C$ are the Lipschitz constants of $\nabla h_\eta$ and $ h^{-1}_\eta$, respectively.
For any $\|y-x\|\geq \delta$,
\begin{align}
\iota_S(y) - \iota_S(x) - \big<d,y - x\big> \nonumber
=& - \big<d,y - x\big> \\ \nonumber
\geq& - \|d\|\cdot \|y - x\|\\
\geq & -\frac{M}{\delta}\|y - x\|^2, \label{Eq:intersection_base_2}
\end{align}
where $M$ is the maximum of $\|d\|$ over $\partial \iota_S(x)$. Combining \eqref{Eq:intersection_base_1} and \eqref{Eq:intersection_base_2} shows that $\iota_{S}$ is restricted prox-regular.
\end{proof}

\begin{proof}[\textbf{Lemma \ref{lemma:reverse_control}}]
By the definitions of $H$ in \ref{A_rank}(a) and $y^k$, we have
  $y^k = H(By^k)$. Therefore,
$\|y^{k_1} - y^{k_2}\|=\|H(By^{k_1}) - H(By^{k_2})\| \leq \bar{M} \|By^{k_1} - By^{k_2}\|.$
Similarly, by the optimality of $x^k_i$, we have $x^k_i = F_i(A_ix_i^k)$. Therefore,
$\|x^{k_1}_i - x_i^{k_2}\|=\|F_i(A_ix_i^{k_1}) - F_i(A_ix_i^{k_2})\| \leq \bar{M} \|A_ix_i^{k_1} - A_ix_i^{k_2}\|.$
\qed
\end{proof}

\begin{proof}[\textbf{Lemma \ref{lemma:well_define}}]
Let us first show that the $y$-subproblem is well defined. To begin with, we will show that $h(y)$ is lower bounded by a quadratic function of $By$:
\[h(y) \geq h(H(0)) - \left(\bar{M}\|\nabla h(H(0))\|\right)\cdot\|By\| -\frac{L_h\bar{M}^2}{2} \|By\|^2. \]
By \ref{A_rank}, we know $h(y)$ is lower bounded by $h(H(By))$:
\[h(y)\geq h(H(By)).\]
Because of \ref{A_h} and \ref{A_rank}, $h(H(By))$ is lower bounded by a quadratic function of $By$:
\begin{align}
h(H(By)) \geq &  h(H(0)) + \big<\nabla h(H(0)),H(By) - H(0)\big>  -\frac{L_h}{2} \|H(By) - H(0)\|^2 \\
\geq & h(H(0)) - \|\nabla h(H(0))\|\cdot\bar{M}\cdot\|By\| -\frac{L_h\bar{M}^2}{2} \|By\|^2
\end{align}
Therefore $h(y)$ is also bounded by the quadratic function:
\[h(y) \geq h(H(0)) - \|\nabla h(H(0))\|\cdot\bar{M}\cdot\|By\| -\frac{L_h\bar{M}^2}{2} \|By\|^2. \]
Recall that $y$-subproblem is to minimize the Lagrangian function w.r.t. $y$, by neglecting other constants, it is equivalent to minimize:
\begin{align}
\argmin & ~ P(y) := h(y) + \big<w^{k} + \beta \vA\vx^+, By\big> + \frac{\beta}{2}\|By\|^2.
\end{align}
Because $h(y)$ is lower bounded by $-\frac{L_h\bar{M}^2}{2}\|By\|^2$, when $\beta > L_h\bar{M}$, $P(y)\rightarrow \infty$ as $\|By\|\rightarrow \infty$.
This shows that $y$-subproblem is coercive with respect to $By$. Because $P(y)$ is  lower semi-continuous and $\argmin h(y) \ \text{s.t.} \ By = u$ has a unique solution for each $u$, the minimal point of $P(y)$ must exist and the $y$-subproblem is well defined.

As for the $x_i$-subproblem, $i = 0,\ldots, p$, ignoring the constants yields
\begin{align}
\argmin\ \Lag_\beta(x^{+}_{<i},x_i,x^{k}_{>i},y^k,w^k)
=\argmin&\ f(x^{+}_{<i},x_i,x^k_{>i}) + \frac{\beta}{2}\|\frac{1}{\beta}w^k + A_{<i}x^+_{<i} + A_{>i}x^k_{>i} + A_ix_i + By^k\|^2\\
=\argmin&\ f(x^{+}_{<i},x_i,x^k_{>i}) + h(u) -  h(u) + \frac{\beta}{2}\|Bu-By^k-\frac{1}{\beta}w^k\|^2.
\end{align}
where $u = H(-A_{<i}x^+_{<i} - A_{>i}x^k_{>i} - A_ix_i)$. The first two terms are coercive bounded because $A_{<i}x^+_{<i} + A_{>i}x^k_{>i} + A_ix_i + Bu = 0$ and \ref{A_coercive}. The third and fourth terms are lower bounded because $h$ is Lipschitz differentiable. Because the function is lower semi-continuous, all the subproblems are well defined. \qed
\end{proof}

\begin{proof}[\textbf{Proposition \ref{prop:example}}]
Define the augmented Lagrangian function to be
\[
\Lag_{\beta}(x,y,w) = x^2 - y^2 + w(x-y) + \frac{\beta}{2} \|x - y\|^2.
\]
It is clear that when $\beta=0$, $\Lag_{\beta}$ is not lower bounded for any $w$.
We are going to show that for any $\beta>2$, the duality gap is not zero.
\[
\inf_{x\in [-1,1],y\in \RR}\sup_{w\in \RR} \Lag_{\beta}(x,y,w)
>
\sup_{w\in \RR}\inf_{x\in [-1,1],y\in \RR} \Lag_{\beta}(x,y,w).
\]
On one hand, because $\sup_{w\in \RR} \Lag_{\beta}(x,y,w) = +\infty$ when $x\neq y$  and $\sup_{w\in \RR} \Lag_{\beta}(x,y,w) = 0$ when $x=y$, we have
\[
\inf_{x\in [-1,1],y\in \RR}\sup_{w\in \RR} \Lag_{\beta}(x,y,w)
= 0.
\]
On the other hand, let $t=x-y$,
\begin{align}
\label{Eq:transform}
&\sup_{w\in \RR}\inf_{x\in [-1,1],y\in \RR} \Lag_{\beta}(x,y,w)
=
\sup_{w\in \RR}\inf_{x\in [-1,1],t\in \RR} t(2x-t) + wt+\frac{\beta}{2} t^2=
\sup_{w\in \RR}\inf_{x\in [-1,1],t\in \RR} (w+2x)t+\frac{\beta-2}{2} t^2\\
&=
\sup_{w\in \RR}\inf_{x\in [-1,1]} -\frac{(w+2x)^2}{2(\beta-2)}=
\sup_{w\in \RR} -\frac{\max\{(w-2)^2,(w+2)^2\}}{2(\beta-2)}=
-\frac{2}{\beta - 2}.
\end{align}
This shows the duality gap is not zero (but it goes to $0$ as $\beta$ tends to $\infty$).

Then let us show that ALM does not converge if $\beta^k$ is bounded, i.e., there exists $\beta>0$ such that $\beta^k\leq \beta$ for any $k\in \mathbb{N}$. Without loss of generality, we  assume that $\beta^k$ equals to the constant $\beta$ for all $k\in \mathbb{N}$. This will not affect the proof. ALM consists of two steps
\begin{enumerate}
\item[1)] $(x^{k+1},y^{k+1}) = \text{argmin}_{x,y} \Lag_{\beta}(x,y,w^k),$
\item[2)] $w^{k+1} = w^k + \tau (x^{k+1} - y^{k+1}).$
\end{enumerate}
Since $(x^{k+1} - y^{k+1})\in \partial \psi(w^k)$ where $\psi(w) = \inf_{x,y} \Lag_{\beta}(x,y,w)$, and we already know
\[\inf_{x,y} \Lag_{\beta}(x,y,w) = -\frac{\max((w-2)^2,(w+2)^2)}{2(\beta-2)},\]
we have
\[
w^{k+1} =
\left\{
\begin{array}{cc}
(1-\frac{\tau}{\beta-2}) w^k -  \frac{2\tau}{\beta-2} & \text{ if } w^{k} \geq 0\\
(1-\frac{\tau}{\beta-2}) w^k +  \frac{2\tau}{\beta - 2} & \text{ if } w^{k} \leq 0
\end{array}
\right. .
\]
Note that when $w^k = 0$, the optimization problem $\inf_{x,y} L(x,y,0)$ has two distinct minimal points which lead to two different values. This shows no matter how small $\tau$ is, $w^k$ will oscillate around $0$ and never converge.

However, although the duality gap is not zero, ADMM still converges in this case. There are two ways to prove it.
The first way is to check all the conditions in Theorem \ref{thm:main}. Another way is to check the iterates directly.
The ADMM iterates are
\begin{align}
x^{k+1} = \max\left(-1,\min(1,\frac{\beta}{\beta + 2}(y^k - \frac{w^k}{\beta}))\right), \quad
y^{k+1} = \frac{\beta}{\beta - 2}\big(x^{k+1}+\frac{w^k}{\beta}\big),
\quad w^{k+1} = w^k + \beta(x^{k+1} - y^{k+1}).
\end{align}
The second equality shows that $w^{k} = -2y^k$, substituting it into the first and second equalities, we have
\begin{align}
x^{k+1} = \max\{-1,\min\{1,y^k\}\},\quad
y^{k+1} = \frac{1}{\beta - 2}\left(\beta x^{k+1} - 2y^k\right).
\end{align}
Here $|y^{k+1}| \leq \frac{\beta}{\beta-2} + \frac{2}{\beta-2}|y^k|$.
Thus after finite iterations, $|y^{k}| \leq 2$ (assume $\beta>4$).
If $|y^k| \leq 1$, the ADMM sequence converges obviously.
If $|y^k| > 1$, without loss of generality, we could assume $2>y^k>1$.
Then $x^{k+1} = 1$.
It means $0<y^{k+1}<1$, so the ADMM sequence converges.
Thus, we know for any initial point $y^0$ and $w^0$, ADMM
converges.
\end{proof}
\medskip

\begin{proof}[\textbf{Theorem \ref{thm:main_theorem_convex}}]
Similar to the proof of Theorem \ref{thm:main}, we only need to verify P1-P4 in Proposition \ref{prop:convergence}.
\emph{Proof of P2: }
Similar to Lemma \ref{lm:x} and Lemma \ref{lm:yw}, we have
\begin{align}\nonumber
& \Lag_\beta(\vx^{k},y^k,w^k)-\Lag_\beta(\vx^{k+1},y^{k+1},w^{k+1})   \\
\geq&-\frac{1}{\beta}\|w^{k} - w^{k+1}\|^2 + \sum_{i=1}^p\frac{\beta - L_\phi\bar{M}}{2}\|A_ix_i^k-A_ix_i^{k+1}\|^2 + \frac{\beta - L_\phi\bar{M}}{2}\|By^k - By^{k+1}\|^2.
\end{align}
Since $B^Tw^k = - \partial_y \phi(\vx^k,y^k)$ for any $k\in \mathbb{N}$, we have
\[\|w^k - w^{k+1}\|\leq C_1L_\phi \left(\sum_{i=0}^p \|x_i^k - x_i^{k+1}\| + \|y^k - y^{k+1}\|\right),\]
where $C_1 = \sigma_{\min}(B)$, $\sigma_{\min}(B)$ is the smallest positive singular value of $B$, and $L_\phi$ is the Lipschitz constant for $\phi$. Therefore, we have
\begin{align}\nonumber
& \Lag_\beta(\vx^{k},y^k,w^k)-\Lag_\beta(\vx^{k+1},y^{k+1},w^{k+1})   \\
\geq& \left(\frac{\beta - L_\phi\bar{M}}{2} - \frac{CL_\phi\bar{M}}{\beta}\right)\sum_{i=0}^p\|A_ix_i^k-A_ix_i^{k+1}\|^2 + \left(\frac{\beta - L_\phi\bar{M}}{2}-\frac{C_1L_\phi\bar{M}}{\beta}\right)\|By^k - By^{k+1}\|^2.
\end{align}
When $\beta > \max\{1,L_\phi \bar{M} + 2C_1L_\phi \bar{M}\}$, P2 holds.

\emph{Proof of P1:}
First of all, we have already shown $\Lag_\beta(\vx^{k},y^k,w^k)\geq \Lag_\beta(\vx^{k+1},y^{k+1},w^{k+1})$, which means $\Lag_\beta(\vx^{k},y^k,w^k)$ decreases monotonically. There exists $y'$ such that $\vA\vx^k + By' = 0$ and $y' = H(By')$. In order to show $\Lag_\beta(\vx^k,y^k,w^k)$ is lower bounded, we apply \ref{A_coercive}-\ref{A_rank} to get
\begin{align}\label{Eq:boundedness_coupled}
   & \Lag_\beta(\vx^{k},y^k,w^k)=\phi(\vx^{k},y^k)+\big<w^{k},\sum_{i=0}^p A_ix^{k}_i + By^k\big>+ \frac{\beta}{2}\|\sum_{i=0}^p A_ix^{k}_i + By^k\|^2\\
   & = \phi(\vx^{k},y^k)+\big<d_y^k,y' - y^k\big>+ \frac{\beta}{2}\| By^k - By'\|^2
    \geq \phi(\vx^{k},y')+\frac{\beta}{4}\|\sum_{i=0}^p A_ix^{k}_i + By^k\|^2\nonumber
    > -\infty,\nonumber
 \end{align}
 for some $d_y^k \in \partial_y \phi(\vx^{k},y^k)$.
This shows that ${\cal L}_{\beta}(\vx^{k},y^k,w^k)$ is lower bounded. If we view \eqref{Eq:boundedness_coupled} from the opposite direction, it can be observed that
  \[\phi(\vx^k,y')+\frac{\beta}{4}\|\sum_{i=1}^p A_ix^{k}_i + By^k\|^2\]
is upper bounded by $\Lag_\beta(\vx^0,y^0,w^0)$. Then \ref{A_coercive} ensures that $\{\vx^k,y^k\}$ is bounded. Therefore, $w^k$ is bounded too.

\emph{Proof of P3, P4:} This part is trivial as $\phi$ is Lipschitz differentiable. Hence we omit it.
\end{proof}

\end{document}

%% file: macros.tex
\usepackage{mathtools}
\usepackage{graphicx}
\usepackage{amssymb}
\usepackage{algorithm,algorithmic}
\usepackage{enumitem} 

\usepackage[normalem]{ulem} 

\newcommand\cut[1]{{}}

\newcommand{\vx}{{\mathbf{x}}}

\newcommand{\vA}{{\mathbf{A}}}

\newcommand{\cB}{{\mathcal{B}}}

\newcommand{\cF}{{\mathcal{F}}}

\newcommand{\RR}{\mathbb{R}}

\newcommand{\St}{{\mathrm{subject~to}}} 
\newcommand{\dom}{{\mathrm{dom}}} 
\newcommand{\Proj}{{\mathrm{Proj}}}

\newcommand{\image}[1]{\mathrm{Im}(#1)}
\DeclareMathOperator*{\argmin}{argmin}

\DeclareMathOperator*{\Min}{minimize}

\DeclarePairedDelimiter{\dotp}{\langle}{\rangle}

\newcommand{\bc}{\begin{center}}
\newcommand{\ec}{\end{center}}

\newcommand{\bdm}{\begin{displaymath}}
\newcommand{\edm}{\end{displaymath}}

\newcommand{\beq}{\begin{equation}}
\newcommand{\eeq}{\end{equation}}

\newcommand{\bfl}{\begin{flushleft}}
\newcommand{\efl}{\end{flushleft}}

\newcommand{\bt}{\begin{tabbing}}
\newcommand{\et}{\end{tabbing}}

\newcommand{\beqn}{\begin{align}}
\newcommand{\eeqn}{\end{align}}

\newcommand{\beqs}{\begin{align*}} 
\newcommand{\eeqs}{\end{align*}}  

